\documentclass{amsart}

\usepackage{amssymb,color}
\usepackage{amsfonts}
\usepackage{amsmath}
\usepackage{euscript}
\usepackage{enumerate}
\usepackage[utf8]{inputenc}

\usepackage{pdfsync}
\newtheorem{theorem}{Theorem}[section]
\newtheorem{lemma}[theorem]{Lemma}

\newtheorem{prop}[theorem]{Proposition}
\newtheorem{cor}[theorem]{Corollary}

\newtheorem*{Theorem1'}{Theorem 1'}

\theoremstyle{definition}

\theoremstyle{remark}

\newcommand \GL{{\mathrm{GL}}}
\newcommand \Z{{\mathbb Z}}

\newcommand \N{{\mathbb N}}

\begin{document}

\title[The automorphism group of finite $p$-groups associated to the Macdonald group]{The automorphism group of finite $p$-groups associated to the Macdonald group}

\author{Alexander Montoya Ocampo}
\address{Department of Mathematics and Statistics, University of Regina, Canada}
\email{alexandermontoya1996@gmail.com}

\author{Fernando Szechtman}
\address{Department of Mathematics and Statistics, University of Regina, Canada}
\email{fernando.szechtman@gmail.com}
\thanks{The second author was partially supported by NSERC grant 2020-04062}





\begin{abstract} Given positive integers $p$ and $m$, where $p$ is assumed to be an odd prime, we determine
the automorphism groups of $p$-groups $J$, $H$, and $K$ of orders $p^{7m}$, $p^{6m}$, and $p^{5m}$, and nilpotency
classes 5, 4, and 3, respectively, all of which arise naturally from the Macdonald group
$\langle x,y\,|\, x^{[x,y]}=x^{1+p^m\ell},\, y^{[y,x]}=y^{1+p^m\ell}\rangle$, where $\ell\in\Z$ is relatively prime to $p$.
\end{abstract}

\maketitle

\section{Introduction}

We fix an odd prime $p\in\N$ and a positive integer $m$ throughout the paper. Given an integer $\ell$ relatively prime to $p$
such that $(p,m,[\ell]_p)\neq (3,1,[-1]_3)$, we consider
the following groups:
$$
J=\langle x,y\,|\, x^{[x,y]}=x^{1+p^m\ell},\, y^{[y,x]}=y^{1+p^m\ell}, x^{p^{3m}}=y^{p^{3m}}=1\rangle,
$$
$$
H=J/Z(J)=\langle x,y\,|\, x^{[x,y]}=x^{1+p^m\ell},\, y^{[y,x]}=y^{1+p^m\ell}, x^{p^{2m}}=y^{p^{2m}}=1\rangle,
$$
$$
K=H/Z(H)=\langle x,y\,|\, x^{[x,y]}=x^{1+p^m\ell},\, y^{[y,x]}=y^{1+p^m\ell}, x^{p^{2m}}=y^{p^{2m}}=[x,y]^{p^m}=1\rangle.
$$
Each of $J$, $H$, and $K$ is a finite $p$-group and in this paper we determine their automorphism groups.

Given a group $T$ and $i\geq 0$, we let $\langle 1\rangle=Z_0(T),Z_1(T),Z_2(T),\dots$ stand for the terms of the upper central series of~$T$, so that $Z_{i+1}(T)/Z_i(T)$ is the center of $T/Z_i(T)$, and we write $\mathrm{Aut}_i(T)$ for the kernel of the canonical map
$\mathrm{Aut}(T)\to \mathrm{Aut}(T/Z_i(T))$.
Thus, each $\mathrm{Aut}_i(T)$ is a normal subgroup of $\mathrm{Aut}(T)$. If $T$ is nilpotent of class $c$, we then
have the normal series
\begin{equation}
\label{norser}
\mathrm{Aut}_0(T)=1\subseteq \mathrm{Aut}_1(T)\subseteq\dots\subseteq \mathrm{Aut}_c(T)=\mathrm{Aut}(T).
\end{equation}

Take $T$ to be any of the groups $J,H,K$. Then $c=5$ if $T=J$, so $c=4$ if $T=H$ and $c=3$ if $T=K$. Whatever the case,
we determine the isomorphism type of each factor $\mathrm{Aut}_{i+1}(T)/\mathrm{Aut}_i(T)$ of (\ref{norser}) for all $0\leq i<c$.
The order of $\mathrm{Aut}(T)$ is obtained as an immediate consequence. We also use the subgroups
$\mathrm{Aut}_i(T)$ to obtain explicit generators of $\mathrm{Aut}(T)$. The group $\mathrm{Aut}(T)$
has a normal Sylow $p$-subgroup, say $\mathrm{Aut}(T)_p$, and the corresponding factor group is rather small. In fact,
$\mathrm{Aut}(T)/\mathrm{Aut}(T)_p$ has order~2 if $T=J$ or $T=H$, and is isomorphic to the dihedral group of order $2(p-1)$
(this reduces to the Klein 4-group if $p=3$) if $T=K$. We also give explicit generators and defining relations (i.e. a presentation)
for $\mathrm{Aut}(T)_p$. The above is successively done for $K$, $H$, and $J$, each case being used to produce the next one.

In order to understand the automorphism group of a given group one must first know its structure. In our
case, the structure of $J$ is elucidated in \cite{MS} under no restrictions on the given parameters. 
We rely heavily on this source, which readily yields the
structures of $H$ and $K$, as indicated below.

Given integers $\alpha$ and $\beta$, the group
$$
G(\alpha,\beta)=\langle x,y\,|\, x^{[x,y]}=x^{\alpha},\, y^{[y,x]}=y^{\beta}\rangle
$$
was studied by Macdonald \cite{M}. He showed that $G(\alpha,\beta)$ is finite provided $\alpha$ and $\beta$ are different from one,
in which case the prime factors of $|G(\alpha,\beta)|$ coincide with those of $(\alpha-1)(\beta-1)$.

Being a finite group on 2 generators with 2 defining relations makes $G(\alpha,\beta)$, $\alpha\neq 1\neq\beta$,
a finite group of deficiency zero. Such groups have been studied by several authors.
The first finite group requiring 3 generators and 3 defining relations was 
$M(a,b,c)=\langle x,y,z\,|\, x^y=x^a, y^{z}=y^b, z^x=z^c\rangle$, under suitable restrictions on $a,b,c$,
found by Mennicke \cite{Me} in 1959. It was intensively studied in \cite{A,AA,Ja,JR,S,W}.
In spite of this scrutiny, the order of $M(a,b,c)$ is known only in special cases. We want
to contribute to the problem of determining 
the structure of other 
finite groups of deficiency zero, of which there are many \cite{Al,AS,AS2,CR,CR2, CRT, J, K, P, R, W2, W3, W4, W5}. 
Our first step in this direction was taken in \cite{MS}, where we filled a gap in Macdonald's proof
of the nilpotency of $G=G(\alpha,\alpha)$ and we elucidated the structure of this group in detail.
The information about $G$ accumulated in \cite{MS} plays an essentially role
in the study of the more elaborate Wamsley groups \cite{W2}, as seen in \cite{PS}. This information also led us to analyze $\mathrm{Aut}(G)$
or, equivalently, the automorphism group of the Sylow subgroups of $G$, when $\alpha\neq 1$. 
The case of the Sylow 2-subgroup of $G$ is carried out in \cite{MS2}.
If $\alpha=1+p^m\ell$, where $p$, $m$, and  $\ell$ are as indicated above, 
then the Sylow $p$-subgroup of $G$ is none other than $J$. 


One incentive for our investigation of the automorphism group of the Sylow subgroups of $G$ when $\alpha\neq 1$ were the detailed and recent studies of the automorphism groups of other families of finite groups of prime power order, as found in \cite{BC, C, C2, Ma, Ma2}, for
instance, all of which are concerned with metacyclic groups. The structure of $J$ is considerably more complicated. 
Another motivation we had for our study was to apply it to
the question that arises by replacing $\alpha$ with another integer $\alpha'\neq 1$: when are the Sylow
subgroups of $G(\alpha,\alpha)$ and $G(\alpha',\alpha')$ corresponding to the same prime isomorphic? This turns
out to be a nontrivial isomorphism problem involving finite groups of prime power order, settled in \cite{MS3},
where the automorphism groups of $J$, $H$, and~$K$ play a crucial role. Yet
another reason to produce this paper is that our strategy to approach $\mathrm{Aut}(T)$, when $T$ is $J$, $H$, or~$K$,
may be of use
to study the automorphism groups of other finite nilpotent groups $T$. This strategy
essentially consists of understanding the factors $\mathrm{Aut}_{i+1}(T)/\mathrm{Aut}_i(T)$ of (\ref{norser}) for all $0\leq i<c$.
An imbedding tool, namely Proposition \ref{zi2}, 
allows us to derive information from the foregoing sections of (\ref{norser}) arising from $T/Z(T)$ to those arising from $T$.
Thus, we begin our work with $T=K=H/Z(H)$, continue with $T=H=J/Z(H)$, and culminate it with $T=J$. 

For $T=K$, we have $c=3$, and $\mathrm{Aut}_{2}(K)\cong Z_2(K)\times Z_2(K)$
via a natural map, as indicated in Proposition \ref{autk2}. Here $Z_2(K)\cong (\Z/p^m\Z)^3$, as given in Section \ref{autgmod2}.
Regarding $\mathrm{Aut}(K)/\mathrm{Aut}_2(K)$, we can view $K/Z_2(K)$ as a free module of rank 2 over $\Z/p^m\Z$.
As $\mathrm{Aut}_2(K)$ is the kernel of the natural map $\mathrm{Aut}(K)\to \mathrm{Aut}(K/Z_2(K))$,
we can view $\mathrm{Aut}(K)/\mathrm{Aut}_2(K)$ as a subgroup of $\GL_2(\Z/p^m\Z)$. A few samples
suggest that this subgroup is the dihedral group $D_{2p^{m-1}(p-1)}$ of order $2p^{m-1}(p-1)$ (which becomes
the Klein 4-group if $p=3$ and $m=1$). That this is always the case is proven in Theorem \ref{autk6}.

For $T=H$, we have $c=4$, and $\mathrm{Aut}_{3}(H)\cong Z_3(H)\times Z_3(H)$
via a natural map, as indicated in Proposition \ref{auth2}. Here $Z_3(H)\cong (\Z/p^m\Z)^2\rtimes \Z/p^{2m}\Z$, as given in Section 
\ref{autgmod56}. Regarding $\mathrm{Aut}(H)/\mathrm{Aut}_3(H)$, we can view it as a subgroup of 
$\mathrm{Aut}(K)/\mathrm{Aut}_2(K)\cong D_{2p^{m-1}(p-1)}$ by Proposition~\ref{zi2}. That in fact
$\mathrm{Aut}(H)/\mathrm{Aut}_3(H)\cong \Z/2\Z$ is proven in Theorem \ref{auth6}.

For $T=J$, we have $c=5$, and $\mathrm{Aut}_{2}(J)\cong Z_2(J)\times Z_2(J)$
via a natural map, as indicated in Proposition \ref{autg}. Here $Z_2(J)\cong \Z/p^m\times \Z/p^{m}\Z$, as given in Section 
\ref{sejotar}. Regarding $\mathrm{Aut}_3(J)/\mathrm{Aut}_2(J)$, applying Proposition \ref{zi2} twice, we obtain
\begin{equation}
\label{meter}
\mathrm{Aut}_3(J)/\mathrm{Aut}_2(J)\hookrightarrow
\mathrm{Aut}_2(H)/\mathrm{Aut}_1(H)\hookrightarrow
\mathrm{Aut}_1(K),
\end{equation}
where $\mathrm{Aut}_1(K)\cong Z(K)\times Z(K)$ via the natural map from Proposition \ref{autk}. Here $Z(K)\cong (\Z/p^m\Z)^2$,
as indicated in Section \ref{autgmod2}. A careful look at the imbedding (\ref{meter}) reveals
that its image, and hence $\mathrm{Aut}_3(J)/\mathrm{Aut}_2(J)$, is isomorphic to $(\Z/p^m\Z)^2$.
This takes considerable effort and is proven in Theorem \ref{autj3}. Now 
$\mathrm{Aut}_4(J)/\mathrm{Inn}(J)\mathrm{Aut}_3(J)$ imbeds into
$\mathrm{Aut}_3(H)/\mathrm{Inn}(H)\mathrm{Aut}_2(H)$, also by Proposition \ref{zi2}. But
the latter group is easily seen to be trivial in Proposition \ref{auth2}. Thus our imbedding
tool yields $\mathrm{Aut}_4(J)=\mathrm{Inn}(J)\mathrm{Aut}_3(J)$. A final use
of Proposition \ref{zi2} yields that $\mathrm{Aut}(J)/\mathrm{Aut}_4(J)\hookrightarrow \mathrm{Aut}(H)/\mathrm{Aut}_3(H)$.
As indicated above, the latter group has order 2, and hence so does the former via a natural automorphism of $J$
not belonging to $\mathrm{Aut}_4(J)$.


It is easy to see that the isomorphism type of $J$ remains invariant throughout $[\ell]_{p^{2m}}$.
Thus, we may suppose that $\ell>0$ and hence $\alpha=1+p^m\ell>1$. Accordingly, we fix an integer
$\alpha>1$ such that $v_p(\alpha-1)=m$ for the remainder of the paper, assuming only that if $p=3$ and $m=1$ then
$(\alpha-1)/3\equiv 1\mod 3$.

Our results are theoretical and computer-free, although several cases of $p$, $m$, and $\ell$
have been successfully tested with Magma,
which has an algorithm to construct the automorphism group of finite $p$-groups based on \cite{ELO}. 

In terms of notation, function composition proceeds from left to right. In agreement with this convention,
if $V$ is a module over a commutative ring $R$ with identity and $V$ admits a finite basis $\{v_1,\dots,v_n\}$,
in order to make the correspondence between $\mathrm{Aut}(V)$ and $\GL_n(R)$ an isomorphism, we will construct
the matrix of a given automorphism of $V$ row by row instead of column by column. Given a group $T$, with $a,b\in T$,
we adopt the following conventions:
$$
[a,b]=a^{-1}b^{-1}ab,\; b^a=a^{-1}ba,\; {}^a b=aba^{-1},
$$
recalling that
\begin{equation}\label{comfor} [a,bc]=[a,c][a,b]^c,\; [bc,a]=[b,a]^c\; [c,a].
\end{equation}

We write $\delta:T\to\mathrm{Aut}(T)$ for the canonical map $a\mapsto a\delta$, where $a\delta$ is conjugation by $a$,
namely the map $b\mapsto b^a$. The automorphisms of $T$ belonging to $\mathrm{Aut}_i(T)$ will be said to be $i$-central,
while those in $\mathrm{Aut}_1(T)$ will simply be said to be central. Note that central
and inner automorphisms commute with each other. Observe also that for $a\in T$, $a\delta\in \mathrm{Aut}_i(T)$ if and only if
$a\in Z_{i+1}(T)$, so that $\mathrm{Inn}(T)\cap \mathrm{Aut}_i(T)=Z_{i+1}(T)\delta\cong Z_{i+1}(T)/Z(T)$.

 Given a ring $R$ with $1\neq 0$, we write $\mathrm{Heis}(R)$ for the Heisenberg group over $R$.

\section{The automorphism group of $J/Z_2(J)$}\label{autgmod2}

We set $K(\alpha)=K=J/Z_2(J)$ throughout this section. According to \cite[Theorem 8.1]{MS}, $J$ is nilpotent of class 5,
whence $K$ is nilpotent of class 3. By \cite[Theorem 5.3]{MS}, $J$ has presentation $J=\langle A,B\,|\,
A^{[A,B]}=A^\alpha, B^{[B,A]}=B^\alpha, A^{p^{3m}}=1=B^{p^{3m}}\rangle$, and we set $C=[A,B]$. 

We write $\rho: J\to K$
for the canonical projection, and set $a=A\rho$, $b=B\rho$, and $c=C\rho$. By \cite[Theorem 8.1]{MS},
we have $Z_2(J)=\langle A^{p^{2m}},C^{p^{m}}\rangle$, where $A^{p^{2m}}B^{p^{2m}}=1$ as seen in~\cite[Section~6]{MS}.
By \cite[Theorem 5.1]{MS}, $K$ is generated by $a$ and $b$, subject to the defining relations
$$
a^{[a,b]}=a^\alpha, b^{[b,a]}=b^\alpha, a^{p^{2m}}=1, b^{p^{2m}}=1, [a,b]^{p^m}=1.
$$
We have an automorphism $a\leftrightarrow b$, say $\mu$, of $K$.
The proof of \cite[Theorem 8.1]{MS} shows that $|K|=p^{5m}$, that
$a,b,c$ have respective orders $p^{2m},p^{2m},p^m$, and that
every element of $K$ can be written in one and only one way as product of elements taken from $\langle a\rangle,
\langle b\rangle$, and $\langle c\rangle$, in any fixed order. As $Z_3(J)=\langle  A^{p^{m}}, B^{p^{m}}, C^{p^{m}}\rangle$ and
$Z_4(J)=\langle  A^{p^{m}}, B^{p^{m}}, C\rangle$ by \cite[Theorem 8.1]{MS}, it follows that
$Z(K)=\langle a^{p^{m}}, b^{p^{m}}\rangle$ as well as $Z_2(K)=\langle a^{p^{m}}, b^{p^{m}}, c\rangle$. The normal
form of the elements of $K$ yields $Z(K)\cong (\Z/p^m\Z)^2$ and $Z_2(K)\cong (\Z/p^m\Z)^3$. As shown in the proof of \cite[Theorem 8.1]{MS},
we also have
$\mathrm{Inn}(K)\cong K/Z(K)\cong J/Z_3(J)\cong \mathrm{Heis}(\Z/p^m\Z)$.

\begin{prop}\label{autk} For every $u,v\in Z(K)$ the assignment
$$
a\mapsto au,\; b\mapsto bv
$$
extends to a central automorphism $(u,v)\Omega$ of $K$ that fixes $Z_2(K)$ pointwise. Moreover, the corresponding map
$\Omega:Z(K)\times Z(K)\to \mathrm{Aut}(K)$ is a group monomorphism whose image is $\mathrm{Aut}_1(K)$. In particular,
$|\mathrm{Aut}_1(K)|=p^{4m}$.
\end{prop}

\begin{proof} Let $u,v\in Z(K)$. Then $[au,bv]=[a,b]=c$ by (\ref{comfor}). As $Z(K)$ has exponent $p^{m}$,
the defining relations of $K$ are preserved. This yields
an endomorphism $(u,v)\Omega$ of $K$. Note that $(u,v)\Omega$ fixes $Z_2(K)$ pointwise, because $c\mapsto [au,bv]=c$, and
$$
a^{p^{m}}\mapsto (au)^{p^{m}}=a^{p^{m}}u^{p^{m}}=a^{p^{m}},\;
b^{p^{m}}\mapsto (bv)^{p^{m}}=b^{p^{m}}v^{p^{m}}=b^{p^{m}}.
$$

It is now easy to see that $(u,v)\Omega\circ (u',v')\Omega=(u,v)(u',v')\Omega$. As $(1,1)\Omega=1_K$,
each $(u,v)\Omega$ is an automorphism of $K$. It is clear that $(u,v) \Omega$ is trivial only
when $(u,v)$ is trivial. Finally, by definition, every central automorphism of $K$ must be of the form $(u,v) \Omega$
for some $u,v\in Z(K)$.
\end{proof}

\begin{prop}\label{autk2} We have $\mathrm{Inn}(K)\cap\mathrm{Aut}_1(K)=\langle c\delta\rangle$,
$|\mathrm{Inn}(K)\cap\mathrm{Aut}_1(K)|=p^{m}$,
$$
\mathrm{Inn}(K)\mathrm{Aut}_1(K)=\mathrm{Aut}_2(K),\; |\mathrm{Aut}_2(K)|=p^{6m}.
$$
If $u,v\in Z_2(K)$, the assignment $a\mapsto au,b\mapsto bv$ extends to a 2-central automorphism $\Gamma_{(u,v)}$ of $K$.
\end{prop}

\begin{proof} Given that $Z_2(K)/Z(K)$ is generated by the coset of $c$ and has order $p^{m}$, we infer that
$Z_2(K)/Z(K)\cong\mathrm{Inn}(K)\cap\mathrm{Aut}_1(K)=\langle c\delta\rangle$ has order $p^{m}$.

We deduce from Proposition \ref{autk} that $|\mathrm{Inn}(K)\mathrm{Aut}_1(K)|=p^{6m}$.
On the other hand, it is obvious  that $\mathrm{Aut}_1(K)\subset\mathrm{Aut}_2(K)$.
As $b^a=bc^{-1}$, $a^b=ac$, with $c\in Z_2(K)$, we have $\mathrm{Inn}(K)\subset\mathrm{Aut}_2(K)$ as well.
The very definition of $\mathrm{Aut}_2(K)$ forces $|\mathrm{Aut}_2(K)|\leq |Z_2(K)|\times |Z_2(K)|=p^{6m}$,
so $\mathrm{Aut}_2(K)=\mathrm{Inn}(K)\mathrm{Aut}_1(K)$ has order $p^{6m}$ and every possible pair
$(u,v)\in Z_2(K)\times Z_2(K)$ gives rise to a 2-central automorphism $a\mapsto au,b\mapsto bv$ of $K$.
\end{proof}

\begin{prop}\label{autk3} For $r,s\in\Z$ such that $rs\equiv 1\mod p^{2m}$, the assignment
\begin{equation}
\label{defrs}
a\mapsto a^r,\; b\mapsto b^s
\end{equation}
extends to an automorphism $f_r$ of $K$. The corresponding map
$f:[\Z/p^{2m}\Z]^\times \to \mathrm{Aut}(K)$ is a group monomorphism, whose image will be denoted by $S$.
\end{prop}

\begin{proof} Since $[a,b]\in Z_2(K)$, (\ref{comfor}) yields
\begin{equation}
\label{rse}
[a^r,b^s]\equiv [a,b]^{rs}\equiv [a,b]\mod Z(K).
\end{equation}
This implies that the first two defining relations of $K$ are preserved. The next two are obviously preserved.
The last defining relation of $K$ is also preserved by (\ref{rse}) and the fact that $Z(K)$ has exponent $p^{m}$.
This produces an endomorphism $f_r$ of $K$. As $f$ is clearly
an homomorphism, $f_r$ has inverse $f_s$ and is then an automorphism. Evidently, $f$ is a monomorphism.
\end{proof}

Thus $|S|=\varphi(p^{2m})=p^{2m-1}(p-1)=p^{m}\varphi(p^{m})$,  where $\varphi$ is Euler's totient function.

\begin{prop}\label{autk4} We have $f_r\in \mathrm{Aut}_2(K)$ if and only if $r\equiv 1\mod p^{m}$. Moreover,
$$
\mathrm{Aut}_2(K)\cap S=\langle f_{1+p^m}\rangle,\; |\mathrm{Aut}_2(K)\cap S|=p^{m},\;
\mathrm{Aut}_2(K) S/\mathrm{Aut}_2(K)\cong [\Z/p^{m}\Z]^\times.
$$
\end{prop}

\begin{proof} Note that $f_r\in \mathrm{Aut}_2(K)$ if and only if
$a^{r-1},b^{s-1}\in Z_2(K)=\langle a^{p^{m}}, b^{p^{m}}, c\rangle$. The normal form of the elements of $K$ makes
the last condition equivalent to $r\equiv 1\mod p^{m}$.

Let $T$ be the kernel of the canonical group epimorphism $[\Z/p^{2m}\Z]^\times\to [\Z/p^{m}\Z]^\times$.
Thus, $T$ corresponds to $\mathrm{Aut}_2(K)\cap S$ under $f$, whence
$$\mathrm{Aut}_2(K) S/\mathrm{Aut}_2(K)\cong
S/(\mathrm{Aut}_2(K)\cap S)\cong [\Z/p^{2m}\Z]^\times/T\cong [\Z/p^{m}\Z]^\times.$$
Moreover, as $|T|=\varphi(p^{2m})/\varphi(p^{m})=p^{m}$, it follows that $|\mathrm{Aut}_2(K)\cap S|=p^{m}$.
But $f_{1+p^m}$ is in $\mathrm{Aut}_2(K)\cap S$ and has order $p^{m}$, so it generates $\mathrm{Aut}_2(K)\cap S$.
\end{proof}


\begin{lemma}\label{basw} For $i,j\in\N$, the following conjugation formula holds in $K$:
$$
(b^j)^{a^i}= a^{(\alpha-1)ij(i-1)/2} c^{-ij}b^{(\alpha-1)ij(j+1)/2}b^j.
$$
\end{lemma}

\begin{proof} We have
\begin{equation}
\label{psq}
(b^j)^{a^i}=(b^{a^i})^j=(b a^{(\alpha-1)i(i-1)/2}c^{-i})^j=a^{(\alpha-1)ij(i-1)/2}(b c^{-i})^j,
\end{equation}
where
\begin{equation}
\label{ys}
(b c^{-i})^j=c^{-ij}b^{\alpha^i(\alpha^{ij}-1)/(\alpha^i-1)}.
\end{equation}
Here
$$
\alpha^{ij}=(1+(\alpha^{i}-1))^{j}=1+j(\alpha^i-1)+{{j}\choose{2}}(\alpha^i-1)^2+{{j}\choose{3}}(\alpha^i-1)^3+\cdots,
$$
so
$$
(\alpha^{ij}-1)/(\alpha^i-1)\equiv j +(\alpha^i-1)j(j-1)/2\mod p^{2m},
$$
and therefore
$$
\alpha^i(\alpha^{ij}-1)/(\alpha^i-1)\equiv j + (\alpha^i-1)j+\alpha^i(\alpha^i-1)j(j-1)/2 \mod p^{2m}.
$$
As $\alpha(\alpha-1)\equiv \alpha-1\mod p^{2m}$ and
$1+\alpha+\cdots+\alpha^{i-1}\equiv i\mod p^m$, we infer
\begin{equation}
\label{xs}
\alpha^i(\alpha^{ij}-1)/(\alpha^i-1)\equiv j + (\alpha-1)ij+(\alpha-1)ij(j-1)/2\equiv j + (\alpha-1)ij(j+1)/2 \mod p^{2m}.
\end{equation}
Replacing (\ref{xs}) in (\ref{ys}) and going back to (\ref{psq}) yields  the desired result.
\end{proof}

\begin{prop}\label{ele} Let $\beta$ be any integer different from 1 such that $v_p(\beta-1)=m$
and if $p=3$ and $m=1$ then $(\beta-1)/3\equiv 1\mod 3$. We then have $K(\alpha)\cong K(\beta)$.
\end{prop}

\begin{proof} We have
$K(\beta)=\langle d,e\,|\, d^{[d,e]}=d^\beta,\, e^{[e,d]}=e^\beta,  d^{p^{2m}}=1=e^{p^{2m}},[d,e]^{p^{m}}=1\rangle$.

Note that $\beta\equiv\alpha^\ell\mod p^{2m}$, where $\ell$ is an integer relatively prime to $p$, uniquely
determined modulo $p^{m}$. Consider the assignment
\begin{equation}\label{pab}
d\mapsto a,\; e\mapsto b^\ell.
\end{equation}
We claim that this extends to an isomorphism $K(\beta)\to K(\alpha)$. Indeed,
$c=[a,b]\in Z_2(K(\alpha))$, so $[a,b^\ell]\equiv [a,b]^\ell \mod Z(K(\alpha))$.
Assuming without loss that $\ell$ is positive,
$$
a^{[a,b^\ell]}=a^{[a,b]^\ell}=a^{\alpha^\ell}=a^\beta,\; b^{[b^\ell,a]}=b^{[b,a]^\ell}=b^{\alpha^\ell}=b^\beta,
$$
so $(b^\ell)^{[b^\ell,a]}=(b^\ell)^\beta$. We clearly also have $a^{p^{2m}}=1=(b^\ell)^{p^{2m}}$. Moreover, $[a,b^\ell]=[a,b]^\ell z$ for some $z\in Z(K(\alpha))$, which is a group of exponent~$p^m$,
so $[a,b^\ell]^{p^m}=1$.
All defining relations of $K(\beta)$ being fulfilled, the assignment (\ref{pab}) extends to a group epimorphism
$K(\beta)\to K(\alpha)$. As both groups have order $p^{5m}$, they are isomorphic.
\end{proof}

By Proposition \ref{ele}, the isomorphism type of $K(\alpha)$ is independent of $\ell$, and we take $\alpha=1+p^m$,
with inverse $\beta=1-p^m$ modulo $p^{2m}$. Then for any positive integer $d$, we have
\begin{equation}
\label{ded}
\alpha^d\equiv 1+d p^m\mod p^{2m}, \beta^d\equiv 1-d p^m \mod p^{2m}.
\end{equation}

Recall that $\mu\in\mathrm{Aut}(K)$ interchanges $a$ and $b$. If $r,s\in\Z$ are inverses modulo $p^{2m}$, then  (\ref{defrs}) gives
\begin{equation}
\label{nors}
(f_r)^\mu=f_s=(f_r)^{-1}.
\end{equation}

We note from (\ref{nors}) that $\mu$ normalizes $S$, so $S\langle \mu\rangle$ is a subgroup of $\mathrm{Aut}(K)$.
As $\mathrm{Aut}_2(K)$ is a normal subgroup of $\mathrm{Aut}(K)$, we infer that $\mathrm{Aut}_2(K)S\langle \mu\rangle$
is a subgroup of $\mathrm{Aut}(K)$ of order $2p^{7m-1}(p-1)$. We are ready to prove the following~result.

\begin{theorem}\label{autk6} The canonical map $\Lambda:\mathrm{Aut}(K)\to \mathrm{Aut}(K/Z_2(K))$
is a group homomorphism with kernel $\mathrm{Aut}_2(K)$ and image $(S\langle\mu\rangle)^\Lambda$, so that
$\mathrm{Aut}(K)=\mathrm{Aut}_2(K)S\langle \mu \rangle$ has order $2p^{7m-1}(p-1)$. Moreover,
$\mathrm{Aut}(K)/\mathrm{Aut}_2(K)\cong D_{2p^{m-1}(p-1)}$, the dihedral group of order $2p^{m-1}(p-1)$ (which becomes
the Klein 4-group if $p=3$ and $m=1$). Furthermore, every element of $\mathrm{Aut}(K)$ can be written in one and only one
way in the form $gf_r\mu^j$, where $g\in \mathrm{Aut}_2(K)$, $0\leq r<p^{m}$ is relatively prime to $p$, and $0\leq j\leq 1$.
\end{theorem}

\begin{proof} Consider the canonical homomorphism $\Lambda:\mathrm{Aut}(K)\to\mathrm{Aut}(K/Z_2(K))$.
Here $K/Z_2(K)$ is a free module over $\Z/p^m\Z$ of rank 2, so $\Lambda$ gives rise to the homomorphism
$D:\mathrm{Aut}(K)\to (\Z/p^m\Z)^\times$, $\Psi\mapsto D_\Psi$, the determinant of $\Psi^\Lambda$.
Let $\Psi\in\mathrm{Aut}(K)$. Then
\begin{equation}
\label{pru}
a^\Psi=a^i b^j z,\; b^\Psi=a^e b^f w\quad i,j,e,f\in\N, z,w\in Z_2(K).
\end{equation}

Let $d$ be any integer that maps into $D_\Psi$ under $\Z\to \Z/p^m\Z$. Since $c\in Z_2(K)$,
we have
\begin{equation}
\label{det}
c^\Psi\equiv [a^i b^j z,a^e b^f w]\equiv c^{if-je}\equiv c^{d}\mod Z(K).
\end{equation}

As $\Psi$ preserves the relation $a^c=a^\alpha$, it follows from (\ref{det}) that
\begin{equation}
\label{compa2}
(a^i b^j z)^{c^d}=(a^i b^j z)^\alpha.
\end{equation}

We may assume without loss that $d$ is positive, so (\ref{ded}) yields
\begin{equation}
\label{comp2}
(a^i b^j z)^{c^d}=a^{i\alpha^d} b^{j\beta^d}z=a^i a^j z a^{i d p^m} b^{-j d p^m}.
\end{equation}

On the other hand,
\begin{equation}
\label{comp3}
(a^i b^j z)^{\alpha}=a^i b^j z(a^i b^j z)^{p^m}.
\end{equation}

From (\ref{compa2})-(\ref{comp3}) we derive
\begin{equation}
\label{coli}
(a^i b^j z)^{p^m}=a^{i d p^m} b^{-j d p^m}.
\end{equation}

Likewise, since $\Psi$ preserves the relation $b^c=b^\beta$, we have $(a^e b^f w)^{c^d}=(a^e b^f w)^\beta$, which leads
to $a^{e d p^m} b^{- f d p^m}=(a^e b^f w)^{-p^m}$, that is,
\begin{equation}
\label{coli2}
(a^e b^f w)^{p^m}=a^{-e d p^m} b^{f d p^m}.
\end{equation}

We next compute $(a^i b^j z)^{p^m}$ and $(a^e b^f w)^{p^m}$ directly, using only $i,j,e,f\in\N$ and $z,w\in Z_2(K)$.
As a first step, we show that
\begin{equation}
\label{powell}
(a^i b^j z)^{p^m}=(a^i b^j)^{p^m}.
\end{equation}
Indeed, we  have $z=a^{u p^m}b^{v p^m}c^k$, $u,v,k\in\N$,
and we set $h=c^k$, $\gamma=\alpha^k$, and $\delta=\beta^k$. Then
$$
\begin{aligned}
(a^i b^j z)^{p^m} & =h^{p^m}(a^i b^j)^{h^{p^m}}\cdots (a^i b^j)^{h^{2}}(a^i b^j)^{h}\\
&=(a^i)^{\gamma^{p^m}} (b^j)^{\delta^{p^m}}\cdots (a^i)^{\gamma^{2}} (b^j)^{\delta^{2}} (a^i)^{\gamma} (b^j)^{\delta}\\
&=(a^i b^j)^{p^m} (a^i)^{(\gamma-1)+(\gamma^2-1)+\cdots+(\gamma^{p^m}-1)}
(b^j)^{(\delta-1)+(\delta^2-1)+\cdots+(\delta^{p^m}-1)}\\
&=(a^i b^j)^{p^m},
\end{aligned}
$$
since $(\gamma^\ell-1)\equiv (\gamma-1)\ell\mod p^{m}$ for any $\ell\in\N$,
so
$$(\gamma-1)+(\gamma^2-1)+\cdots+(\gamma^{p^m}-1)\equiv (\gamma-1)(1+2+\cdots+p^m)\equiv 0\mod p^{2m},$$
and likewise $(\delta-1)+(\delta^2-1)+\cdots+(\delta^{p^m}-1)\equiv 0\mod p^{2m}$.

As a second step, we compute $(a^i b^j)^{p^m}$. We begin by
noting that
\begin{equation}
\label{coli3}
(a^i b^j)^{p^m}=a^{i{p^m}} (b^j)^{a^{(p^m-1)i}}\cdots (b^j)^{a^{2i}}(b^j)^{a^i}b^j.
\end{equation}

To deal with (\ref{coli3}), we appeal to Lemma \ref{basw} to see that for each $1\leq t\leq p^m-1$, we have
\begin{equation}
\label{coli4}
(b^j)^{a^{ti}}= a^{(\alpha-1)tij(ti-1)/2} c^{-tij}b^{(\alpha-1)tij(j+1)/2}b^j.
\end{equation}
Since each of the factors $a^{(\alpha-1)tij(ti-1)/2}, b^{(\alpha-1)tij(j+1)/2}$, $1\leq t\leq p^m-1$, appearing in (\ref{coli3})
is central in $K$, we can collect them as powers of $a$ and $b$, obtaining
\begin{equation}
\label{coli5}
b^{(\alpha-1)ij(j+1)[1+2+\cdots+p^{m}-1]/2}=1,
\end{equation}
\begin{equation}
\label{coli6}
a^{(\alpha-1)j[i(i-1)+2i(2i-1)+\cdots+(p^m-1)i((p^m-1)i-1)]/2}.
\end{equation}
Setting $n=p^m-1$, we claim that
\begin{equation}
\label{freezing}
[i(i-1)+2i(2i-1)+\cdots+ni(ni-1)]/2\equiv\begin{cases} 0\mod p^m & \text{ if }p\neq 3,\\
i^2 3^{m-1}\mod 3^m & \text{ if }p=3.
\end{cases}
\end{equation}
Indeed, we have
$$
i(i-1)+2i(2i-1)+\cdots+ni(ni-1)=-i(1+2+\cdots+n)+i^2(1^2+2^2+\cdots+n^2),
$$
where $(1+2+\cdots+n)\equiv 0\mod p^m$. On the other hand,
$$
1^2+2^2+\cdots+n^2=\frac{n(n+1)(2n+1)}{6},
$$
so $1^2+2^2+\cdots+n^2\equiv 0\mod p^m$ if $p\neq 3$, and
$1^2+2^2+\cdots+n^2\equiv 3^{m-1}(2\times 3^m-1)(3^m-1)/2\equiv -3^{m-1}\mod 3^m$ if $p=3$. This proves the claim.
It follows from (\ref{coli6}) and (\ref{freezing}) that
\begin{equation}
\label{coli7}
a^{(\alpha-1)j[i(i-1)+2i(2i-1)+\cdots+(p^m-1)i((p^m-1)i-1)]/2}=\begin{cases} 1\text{ if }p\neq 3,\\
a^{3^{2m-1} i^2 j}\text{ if }p=3.\end{cases}
\end{equation}

Regarding the right hand side of (\ref{coli3}), except for the leftmost factor $a^{i{p^m}}$
and the collected central factors (\ref{coli5}) and (\ref{coli7}), the remaining factor is equal to
\begin{equation}
\label{coli8}
W=(c^{ij} b^j )(c^{2ij} b^j)\cdots (c^{(p^m-1)ij} b^j) b^j,
\end{equation}
using (\ref{coli4}) and the fact that $c^{p^m}=1$. It follows that
\begin{equation}
\label{coli9}
W=(b^j)^{1+ij{{2}\choose{2}}p^m}(b^j)^{1+ij{{3}\choose{2}}p^m}\cdots (b^j)^{1+ij{{p^m}\choose{2}}p^m} c^{{p^m}\choose{2}} b^j.
\end{equation}
Since $c^{p^m}=1$, we see that
\begin{equation}
\label{coli10}
W=b^{jp^m}b^{ij^2p^m ({{2}\choose{2}}+{{3}\choose{2}}+\cdots+{{p^m}\choose{2}})}.
\end{equation}
Setting $N=p^m$,  we claim that
\begin{equation}
\label{freezing2}
{{2}\choose{2}}+{{3}\choose{2}}+\cdots+{{N}\choose{2}}
\equiv\begin{cases} 0\mod p^m & \text{ if }p\neq 3,\\
3^{m-1}\mod 3^m & \text{ if }p=3
\end{cases}
\end{equation}
Indeed, we have
$$
2({{2}\choose{2}}+{{3}\choose{2}}+\cdots+{{N}\choose{2}})=(2-1)\times 2+(3-1)\times 3+\cdots (N-1)N=
(1^2+2^2+\cdots+N^2)-(1+2+\cdots+N).
$$
Here $1+2+\cdots+N\equiv 0\mod p^m$ and $1^2+2^2+\cdots+N^2=N(N+1)(2N+1)/6$, and therefore
$1^2+2^2+\cdots+N^2\equiv 0\mod p^m$ if $p\neq 3$, and
$1^2+2^2+\cdots+N^2\equiv 3^{m-1}(2\times 3^m+1)(3^m+1)/2\equiv -3^{m-1}\mod p^m$ if $p=3$. This proves the claim.
It follows from (\ref{coli10}) and (\ref{freezing2}) that
\begin{equation}
\label{coli11}
W=b^{ij^2p^m ({{2}\choose{2}}+{{3}\choose{2}}+\cdots+{{p^m}\choose{2}})}=\begin{cases} 1\text{ if }p\neq 3,\\
b^{3^{2m-1} i j^2}\text{ if }p=3.\end{cases}
\end{equation}

Combining (\ref{powell})-(\ref{coli11}), we obtain
\begin{equation}
\label{coli12}
(a^i b^j z)^{p^m}=\begin{cases} a^{ip^m}b^{jp^m}\text{ if }p\neq 3,\\
a^{ip^m}b^{jp^m} a^{3^{2m-1} i^2 j} b^{3^{2m-1} i j^2}\text{ if }p=3.\end{cases}.
\end{equation}

We may also assume without loss that $e$ is even, and, by above, we have
\begin{equation}
\label{coli13}
(a^e b^f w)^{p^m}=\begin{cases} a^{ep^m}b^{fp^m}\text{ if }p\neq 3,\\
a^{ep^m}b^{fp^m} a^{3^{2m-1} e^2 f} b^{3^{2m-1} e f^2}\text{ if }p=3.\end{cases}
\end{equation}

Suppose first that $p\neq 3$. Then (\ref{coli}), (\ref{coli2}), (\ref{coli12}), and (\ref{coli13}) give
\begin{equation}
\label{coli14}
i(d-1)\equiv 0,\, j(d+1)\equiv 0,\, e(d+1)\equiv 0,\, f(d-1)\equiv 0\quad \mod p^m.
\end{equation}
Since $d$ is invertible modulo $p^m$, at least one of $i,j$ must be invertible modulo $p^m$ by (\ref{det}).

If $p\nmid i$ then (\ref{coli14}) yields $d\equiv 1\mod p^m$, whence $2j\equiv 0\mod p^m$, that
is, $j\equiv 0\mod p^m$, and therefore $1\equiv d\equiv if\mod p^m$ and
$2e\equiv 0\mod p^m$, so $e\equiv 0\mod p^m$. If
$p\nmid j$ then replacing $\Psi$ by $\mu\Psi$ the previous case yields $D_\Psi=-1$, $ej\equiv -1\mod p^m$,
and $i,f\equiv 0\mod p^m$.

Suppose next that $p=3$. Then (\ref{coli}), (\ref{coli2}), (\ref{coli12}), and (\ref{coli13})
yield
\begin{equation}
\label{coli15}
i(d-1)-3^{m-1} i^2 j \equiv 0\mod 3^{m},
\end{equation}
\begin{equation}
\label{coli16}
j(d+1)+3^{m-1} i j^2 \equiv 0\mod 3^{m},
\end{equation}
\begin{equation}
\label{coli17}
e(d+1)+3^{m-1} e^2 f \equiv 0\mod 3^{m},
\end{equation}
\begin{equation}
\label{coli18}
f(d-1)-3^{m-1} e f^2 \equiv 0\mod 3^{m}.
\end{equation}

Again, since $3\nmid d$, at least one of $i,j$ must be invertible modulo $3^m$.

Assume first that $3\nmid i$.
Then (\ref{coli15}) yields $d\equiv 1\mod 3^{m-1}$ hence  (\ref{coli16}) ensures that
$j\equiv 0\mod 3^{m-1}$. We continue the argument
assuming $m>1$. We go back to  (\ref{coli15}) and use $j\equiv 0\mod 3^{m-1}$ to obtain
$d\equiv 1\mod 3^m$. This, $j\equiv 0\mod 3^{m-1}$, and (\ref{coli16}) now give $j\equiv 0\mod 3^{m}$.
Next $d\equiv 1\mod 3^m$ and  $j\equiv 0\mod 3^{m}$ yield $if\equiv 1\mod 3^m$. In particular, $3\nmid f$. Also
$d\equiv 1\mod 3^m$ and
(\ref{coli17}) give $e\equiv 0\mod 3^{m-1}$,
which substituted back in (\ref{coli17}) forces $e\equiv 0\mod 3^{m}$. We proceed next assuming $m=1$. In this
case $d\equiv \pm 1\mod 3$ for lack of any other options. We are still assuming $3\nmid i$. Suppose, if possible,
that $d\equiv -1\mod 3$. Then (\ref{coli16}) forces $3|j$ and then $d\equiv 1\mod 3$ by (\ref{coli15}), a contradiction.
Thus $d\equiv 1\mod 3$, so
(\ref{coli15}) implies $3\mid j$, and therefore
$if\equiv 1\mod 3$, whence (\ref{coli18}) yields $3\mid e$.

Suppose next $3\nmid j$. Replacing $\Psi$ by $\mu\Psi$ the above case yields $3\nmid e$, $3^m|i$, $3^m|f$, and
$D_\Psi=-1$.

We have shown that $D_\Psi=\pm 1$, and that if $D_\Psi=1$ then
$if\equiv 1\mod p^m$ and $j,e\equiv 0\mod p^m$,
and if $D_\Psi=-1$ then $je\equiv -1\mod p^m$ and $i,f\equiv 0\mod p^m$.

We fix the $\Z/p^{m}\Z$-basis $\{a Z_2(K), bZ_2(K)\}$ of $K/Z_2(K)$, and identify $\mathrm{Aut}(K/Z_2(K))$
with $\GL_2(\Z/p^m\Z)$. We then have the following matrices:
$$
M_r=f_r^\Lambda=
\left(\begin{array}{cc}
[r] & 0\\ 0 & [s]
\end{array}\right),\,
Q=\mu^\Lambda=
\left(\begin{array}{cc}
0 & 1\\ 1 & 0
\end{array}\right),
$$
where $r\in\Z$, is relatively prime to $p$ with inverse $s$ modulo $p^{2m}$, and $t\mapsto [t]$ is the canonical projection
$\Z\to\Z/p^{m}\Z$. Set $U=\{M_r\,|\, r\in\Z, r\text{ prime to }p\}$ and $V=U\rtimes\langle Q\rangle$. Then $U$
is cyclic of order $\varphi(p^m)=p^{m-1}(p-1)$ and $V$
is a dihedral group of order $2p^{m-1}(p-1)$.

If $D_\Psi=1$ then $if\equiv 1\mod p^m$ and $j,e\equiv 0\mod p^m$, which makes it obvious that
$\Psi^\Lambda\in U$, so if $D_\Psi=-1$ then
$\Psi^\Lambda\in V$. This proves that the image of $\Lambda$ is $V=(S\langle \mu\rangle)^\Lambda$. Since the kernel
of $\Lambda$ is, by definition, $\mathrm{Aut}_2(K)$, it follows that
$\mathrm{Aut}(K)=\mathrm{Aut}_2(K)S\langle \mu\rangle$. As $|\mathrm{Aut}_2(K)|=|Z_2(K)|^2=p^{6m}$
by Proposition \ref{autk2}, and $|V|=2p^{m-1}(p-1)$ by above, we deduce $|\mathrm{Aut}(K)|=2p^{7m-1}(p-1)$.
\end{proof}

\begin{cor} Every automorphism of $K$ has determinant $\pm 1$ when acting on $K/Z_2(K)$,
and $\mathrm{Aut}_2(K)S$ is the kernel of the corresponding determinant map, as well as the
pointwise stabilizer of $c Z(K)$ when $\mathrm{Aut}(K)$ acts on $K/Z(K)$.
\end{cor}

In terms of group extensions, we have the normal series
$$
1\subset \mathrm{Aut}_1(K)\subset \mathrm{Aut}_2(K)\subset \mathrm{Aut}(K).
$$
The first factor $\mathrm{Aut}_1(K)\cong (\Z/p^m\Z)^4$ by Proposition \ref{autk}. By Proposition \ref{autk2}, the second factor
$$
\mathrm{Aut}_2(K)/\mathrm{Aut}_1(K)\cong \mathrm{Aut}_1(K)\mathrm{Inn}(K)/\mathrm{Aut}_1(K)\cong
\mathrm{Inn}(K)/\mathrm{Inn}(K)\cap \mathrm{Aut}_1(K)\cong K\delta/Z_2(K)\delta
$$
is isomorphic to $(K/Z(K))/(Z_2(K)/Z(K))\cong K/Z_2(K)\cong (\Z/p^m\Z)^2$. The last factor
$$
\mathrm{Aut}(K)/\mathrm{Aut}_2(K)\cong D_{2p^{m-1}(p-1)}
$$
by Theorem \ref{autk6}.

Let $S_p$ be the Sylow $p$-subgroup of $S$. As $\mu$ and $S$ normalize $S_p$,
it follows that $\mathrm{Aut}_2(K)S_p$ is a normal subgroup of $\mathrm{Aut}(K)$ of order $p^{7m-1}$, and is therefore the sole
Sylow $p$-subgroup of $\mathrm{Aut}(K)$. We next find a presentation for $\mathrm{Aut}_2(K)S_p$.

In terms of the notation of Proposition \ref{autk}, we set
\begin{equation}\label{pu2}
x=(a^{p^m},a^{p^m})\Omega,\; y=(b^{p^m},b^{p^m})\Omega,\; z=(b^{p^m},a^{p^m})\Omega,\; u=(a^{\alpha-1},b^{1-\alpha})\Omega.
\end{equation}
In this notation, we have the internal direct product decomposition
\begin{equation}
\label{intpro}
\mathrm{Aut}_1(K)=\langle x\rangle\times\langle y \rangle \times
\langle z\rangle\times \langle u\rangle,
\end{equation}
where $u=c\delta$ generates $\mathrm{Aut}_1(K)\cap\mathrm{Inn}(K)$ by Proposition \ref{autk2}.

Recalling that $\mathrm{Aut}_1(K)$ and $\mathrm{Inn}(K)$ commute elementwise,
(\ref{pu2}), (\ref{intpro}), and Proposition \ref{autk2} yield the following in internal direct product decomposition:
\begin{equation}
\label{intpro2}
\mathrm{Aut}_2(K)=\mathrm{Aut}_1(K)\mathrm{Inn}(K)=\langle x\rangle\times\langle y \rangle \times
\langle z\rangle\times \mathrm{Inn}(K).
\end{equation}
Here $x,y,z$ generate a subgroup of $\mathrm{Aut}_2(K)$ isomorphic to $(\Z/p^m\Z)^3$
with defining relations
\begin{equation}\label{pelt2}
x^{p^m}=y^{p^m}=z^{p^m}=1, [x,y]=[x,z]=[y,z]=1,
\end{equation}
and
\begin{equation}\label{pu1}
F=a\delta, G=b\delta
\end{equation}
generate the subgroup $\mathrm{Inn}(K)\cong \mathrm{Heis}(\Z/p^m\Z)$ of $\mathrm{Aut}_2(K)$
subject to the defining relations
\begin{equation}\label{pelt1}
F^{[F,G]}=F, G^{[F,G]}=G, F^{p^{m}}=1, G^{p^{m}}=1,
\end{equation}
which imply the relation $[F,G]^{p^m}=1$. We may now use the decomposition (\ref{intpro2}) to see that $\mathrm{Aut}_2(K)$ is generated by
$x,y,z,F,G$
subject to the defining relations (\ref{pelt2}), (\ref{pelt1}) as well as
\begin{equation}\label{pelt3}
[x,F]=[x,G]=[y,F]=[y,G]=[z,F]=[z,G]=1.
\end{equation}

Now $\mathrm{Aut}_2(K)$ is a normal subgroup of $\mathrm{Aut}_2(K)S_p$
of order $p^{6m}$, and $S_p$ is generated by
\begin{equation}\label{pu3}
w=f_r,\quad r=1+p.
\end{equation}
The order of $w$ modulo $\mathrm{Aut}_2(K)$
is precisely $p^{m-1}$. We have
\begin{equation}\label{q1}
r^{p^{m-1}}\equiv 1+p^m d\mod p^{2m},\quad d\equiv 1\mod p^m.
\end{equation}
On the other hand,
\begin{equation}\label{q2}
\alpha=1+p^m \ell,\quad p\nmid \ell,
\end{equation}
and we select $e\in\Z$ such that
\begin{equation}\label{q3}
de\equiv \ell\mod p^m.
\end{equation}
Thus
$$
r^{p^{m-1}e}\equiv (1+p^m d)^e\equiv 1+p^m  de\equiv 1+p^m \ell\equiv \alpha\mod p^{2m},
$$
and therefore $(f_r)^{p^{m-1} e}=c\delta$, $c=[a,b]$, that is,
\begin{equation}\label{pelt4}
w^{p^{m-1} e}=[F,G].
\end{equation}
We next use the formula $(v\delta)^T=(v^T)\delta$ for $T\in\mathrm{Aut}(K)$ and $v\in K$, applied to
$T=w$ and $v\in \{a,b\}$, and obtain
\begin{equation}\label{pelt5}
F^w=F^{r}, G^w=G^{s},\quad rs\equiv 1\mod p^{m},
\end{equation}
where $s$ may taken to be even by adding $p^m$, if necessary. We then have
\begin{equation}\label{pelt6}
x^w=x, y^w=y, z^w=x^{(r^2-s^2)/2}y^{(s^2-r^2)/2}z^{(r^2+s^2)/2}[F,G]^{l(s^2-r^2)/2},\quad l\ell\equiv 1\mod p^m.
\end{equation}

We have proven the following result.

\begin{theorem} The group $\mathrm{Aut}_2(K)S_p$ is a normal
Sylow $p$-subgroup of $\mathrm{Aut}(K)$ of order $p^{7m-1}$ and index $2(p-1)$ in $\mathrm{Aut}(K)$. It is generated by
$F,G,x,y,z,w$, as defined in (\ref{pu2}), (\ref{pu1}), (\ref{pu3}), subject to the defining relations (\ref{pelt2}), (\ref{pelt1}),
(\ref{pelt3}), (\ref{pelt4}), (\ref{pelt5}),  (\ref{pelt6}), where $e$ is defined in  (\ref{pu3}),
(\ref{q1}), (\ref{q2}), (\ref{q3}).
\end{theorem}


\section{The automorphism group of $J/Z(J)$}\label{autgmod56}

We set $H=J/Z(J)$ throughout this section. According to \cite[Theorem 8.1]{MS}, $J$ is nilpotent of class 5,
whence $H$ is nilpotent of class 4. By means of \cite[Theorem 5.3]{MS}, we see that $J$ has presentation $J=\langle X,Y\,|\,
X^{[X,Y]}=X^\alpha, Y^{[Y,X]}=Y^\alpha, X^{p^{3m}}=1=Y^{p^{3m}}\rangle$, and we set $Z=[X,Y]$.

By \cite[Theorem 8.1]{MS},
we have $Z(J)=\langle X^{p^{2m}}\rangle$, where $X^{p^{2m}}Y^{p^{2m}}=1$ as indicated in \cite[Section 6]{MS}.
We deduce from \cite[Theorem 5.1]{MS} that $H$ has presentation
$$
H=\langle A,B\,|\, A^{[A,B]}=A^\alpha, B^{[B,A]}=B^\alpha, A^{p^{2m}}=1, B^{p^{2m}}=1\rangle,
$$
and we set $C=[A,B]$. We have an automorphism $A\leftrightarrow B$, say $\nu$, of $H$.
The proof of \cite[Theorem 8.1]{MS} shows that $|H|=p^{6m}$, that all of
$A,B,C$ have order $p^{2m}$, and that
every element of $H$ can be written in one and only one way as product of elements taken from $\langle A\rangle,
\langle B\rangle$, and $\langle C\rangle$, in any fixed order. We have $Z_2(J)=\langle  X^{p^{2m}}, Z^{p^{m}}\rangle$,
$Z_3(J)=\langle  X^{p^{m}}, Y^{p^{m}}, Z^{p^{m}}\rangle$, and
$Z_4(J)=\langle  X^{p^{m}}, Y^{p^{m}}, Z\rangle$ by \cite[Theorem 8.1]{MS}. Moreover, $Z_3(J)$ is abelian by
\cite[Proposition 8.1]{MS}. It follows that
$$Z(H)=\langle C^{p^{m}}\rangle, \;
Z_2(H)=\langle A^{p^{m}}, B^{p^{m}}, C^{p^{m}}\rangle,\; Z_3(H)=\langle A^{p^{m}}, B^{p^{m}}, C\rangle,$$
with $Z_2(H)$ abelian. It follows that $Z(H)\cong \Z/p^m\Z$ and $Z_2(H)\cong (\Z/p^m\Z)^3$.
The isomorphism $J/Z_3(J)\cong \mathrm{Heis}(\Z/p^m\Z)$ yields $H/Z_3(H)\cong J/Z_4(J)\cong (\Z/p^m\Z)^2$,
whence $|Z_3(H)|=p^{4m}$.
Also, $\mathrm{Inn}(H)\cong H/Z(H)\cong J/Z_2(J)\cong K$, a group of order $p^{5m}$, where $K$ is defined in Section \ref{autgmod2}.

\begin{prop}\label{auth} For every $x,y\in Z_2(H)$ the assignment
$$
A\mapsto Ax,\; B\mapsto By
$$
extends to a 2-central automorphism $(x,y)\Pi$ of $H$ that fixes $Z_2(H)$ pointwise. Moreover, the corresponding map
$\Pi:Z_2(H)\times Z_2(H)\to \mathrm{Aut}(H)$ is a group monomorphism whose image is $\mathrm{Aut}_2(H)$. In particular,
$|\mathrm{Aut}_2(H)|=p^{6m}$.
\end{prop}

\begin{proof}  We begin by showing that
\begin{equation}
\label{coz2}
(Bu)^{[B,A]}=(Bu)^\alpha,\; (Au)^{[A,B]}=(Au)^\alpha, \;(Bu)^{p^{2m}}=1=(Au)^{p^{2m}},\quad u\in Z_2(H).
\end{equation}

We have
$$
B^{A^{p^{m}}}=BA^{(\alpha-1)(\alpha+2\alpha+\cdots+(p^{m}-1)\alpha^{p^{m}-1})}C^{-p^{m}},
$$
where $\alpha+2\alpha+\cdots+(p^{m}-1)\alpha^{p^{m} -1}\equiv 1+2+\cdots+(p^{m}-1)\equiv 0\mod p^{m}$.
Thus
\begin{equation}
\label{coz3}
B^{A^{p^{m}}}=BC^{-p^{m}}.
\end{equation}
Making use of $C^{p^{m}}\in Z(H)$, we see from (\ref{coz3}) that
\begin{equation}
\label{coz4}
(B^{p^{m}})^{A^{p^{m}}}=(BC^{-p^{m}})^{p^m}=B^{p^{m}},
\end{equation}
\begin{equation}
\label{coz7}
B^{A^{tp^{m}}}=B C^{-tp^{m}},\quad t\in\N.
\end{equation}
We infer from (\ref{coz4}) that $A^{p^{m}}$ and $B^{p^{m}}$ commute with each other,
in agreement with \cite[Proposition 8.1]{MS}, which ensures that $Z_3(J)$, and hence $Z_2(H)$, is abelian. Moreover,
from $\alpha\equiv 1\mod p^m$, we see directly that
\begin{equation}
\label{yug}
[A^{p^m},C]=1=[B^{p^m},C].
\end{equation}

Let $u\in Z_2(H)$. Then $u=A^{ip^{m}}B^{jp^{m}}C^{kp^{m}}$, where $i,j,k\in\N$. As $C^{p^{m}}\in Z(H)$ and
$B^{p^m}$ commutes with $B$ and $A^{p^{m}}$, we see that
\begin{equation}
\label{coz5}
(Bu)^{p^{m}}=(B A^{ip^{m}})^{p^{m}} (B^{jp^{m}}C^{kp^{m}})^{p^{m}}=(B A^{ip^{m}})^{p^{m}}=
A^{i p^{2m}} B^{A^{i p^{2m}}}\cdots B^{A^{2i p^{m}}} B^{A^{i p^{m}}}.
\end{equation}

Appealing to (\ref{coz7}) and (\ref{coz5}), and making use of $C^{p^{m}}\in Z(H)$, we deduce
\begin{equation}
\label{coz8}
(Bu)^{p^{m}}=
B^{p^{m}}C^{-i(1+2+\cdots+p^{m})p^{m}}=B^{p^{m}}.
\end{equation}
Applying the automorphism $A\leftrightarrow B$ of $H$, we infer from (\ref{coz8}) that
\begin{equation}
\label{guy}
(Au)^{p^{m}}=A^{p^{m}}.
\end{equation}

Now $\alpha=1+\ell p^m$ for some $\ell\in\Z$, so (\ref{coz4}), (\ref{yug}), and (\ref{coz8}) yield
\begin{equation}
\label{coz9}
(Bu)^\alpha=Bu(Bu)^{\alpha-1}=BuB^{\ell p^{m}}=B^\alpha u.
\end{equation}

On the other hand, (\ref{yug}) ensures that
\begin{equation}
\label{coz10}
(Bu)^{[B,A]}=B^{[B,A]}u=B^\alpha u.
\end{equation}

We deduce from (\ref{coz9}) and (\ref{coz10}) that $(Bu)^{[B,A]}=(Bu)^\alpha$.
Moreover, (\ref{coz8}) gives $(Bu)^{p^{2m}}=1$. The automorphism $A\leftrightarrow B$ now yields
$(Au)^{[A,B]}=A^\alpha u$ and $(Au)^{p^{2m}}=1$.
This proves that (\ref{coz2}) holds.

Next let $x,y\in Z_2(H)$. Then working modulo $Z(H)$ and making use of (\ref{comfor}), we obtain
\begin{equation}
\label{coz}
[Ax,By]=[A,B]z
\end{equation}
for a unique $z\in Z(H)$. We see from (\ref{coz2}) and (\ref{coz}) that
$$
(By)^{[By,Ax]}=(By)^\alpha,\; (Ax)^{[Ax,By]}=(Ax)^\alpha.
$$
Thus the first two defining relations of $H$ are preserved, and (\ref{coz2}) ensures that the third and fourth
are also preserved.
This yields an endomorphism $(x,y)\Pi$ of $H$. We see from (\ref{coz8}) and (\ref{guy}) that $(x,y)\Pi$
fixes $A^{p^m}$ and $B^{p^m}$. Moreover,
$$
C^{p^{m}}\mapsto (Cz)^{p^{m}}=C^{p^{m}}z^{p^{m}}=C^{p^{m}},
$$
because $z\in\langle Z(H)\rangle$, a group of exponent $p^{m}$.
Thus $(x,y)\Pi$ fixes $Z_2(H)$ pointwise. We may now continue the argument used in
the proof of Proposition \ref{autk}.
\end{proof}

\begin{prop}\label{auth2} We have $\mathrm{Inn}(H)\cap\mathrm{Aut}_2(H)=
\langle C\delta, A^{p^m}\delta, B^{p^m}\delta\rangle$, $|\mathrm{Inn}(H)\cap\mathrm{Aut}_2(H)|=p^{3m}$,
$$
\mathrm{Inn}(H)\mathrm{Aut}_2(H)=\mathrm{Aut}_3(H),\; |\mathrm{Aut}_3(H)|=p^{8m}.
$$
Given any $u,v\in Z_3(H)$, the assignment $A\mapsto Au,B\mapsto Bv$ extends to a 3-central automorphism of $H$.
\end{prop}

\begin{proof} As $Z_3(H)/Z(H)$ has order $p^{3m}$ and is generated by the cosets of $C,A^{p^{m}}, B^{p^{m}}$,
it follows that
$\mathrm{Inn}(H)\cap\mathrm{Aut}_2(H)=\langle C\delta,A^{p^{m}}\delta, B^{p^{m}}\delta\rangle$ has order $p^{3m}$.

We deduce from Proposition \ref{auth} that $|\mathrm{Inn}(H)\mathrm{Aut}_2(H)|=p^{8m}$.
On the other hand, it is obvious  that $\mathrm{Aut}_2(H)$ is included in $\mathrm{Aut}_3(H)$.
As $B^A=BC^{-1}$, $A^B=AC$, with $C\in Z_3(H)$, it follows that $\mathrm{Inn}(H)$ is also included in $\mathrm{Aut}_3(H)$.
The very definition of $\mathrm{Aut}_3(H)$ forces $|\mathrm{Aut}_3(H)|\leq |Z_3(H)|\times |Z_3(H)|=p^{8m}$,
so $\mathrm{Aut}_3(H)=\mathrm{Inn}(H)\mathrm{Aut}_2(H)$ has order $p^{8m}$ and every possible pair
$(u,v)\in Z_3(H)\times Z_3(H)$ gives rise to a 3-central automorphism $A\mapsto Au,B\mapsto Bv$ of $H$.
\end{proof}

As $\nu$ has order 2, we see that that $\nu$ is not in the $p$-group $\mathrm{Aut}_3(H)$. But $\nu$ normalizes $\mathrm{Aut}_3(H)$,
so $\mathrm{Aut}_3(H)\langle \nu\rangle$ is a group of order $2p^{8m}$.

\begin{prop}\label{tecnico} Let $r$, $s$, and $\ell$ be integers relatively prime to $p$ such that $rs\equiv 1\mod p^{2m}$,
and let $\beta$ be an integer such that $\beta\equiv\alpha^\ell\mod p^{2m}$. Suppose that there exist $z_1,z_2\in Z(H)$ such that
\begin{equation}
\label{pli}
(A^r z_1)^{[A^r z_1,B^{s\ell} z_2]}=(A^r z_1)^\beta,\; (B^{s\ell} z_2)^{[B^{s\ell} z_2,A^r z_1]}=(B^{s\ell} z_2)^\beta.
\end{equation}
Then $r\equiv 1\mod p^m$ and $\ell\equiv 1\mod p^m$.
\end{prop}

\begin{proof} Since $z_1,z_2\in Z(H)$, which is a group of exponent $p^{m}$, (\ref{pli}) yields
\begin{equation}
\label{tel0}
(A^r)^{[A^r,B^{s\ell}]}=(A^r)^\beta,\; (B^{s\ell})^{[B^{s\ell},A^r]}=(B^{s\ell})^\beta.
\end{equation}
By adding suitable multiples of $p^{2m}$ to $r$, $s$, and $\ell$, we may assume that all of them are odd,
$r,\ell\geq 1$, and $s<0$, so that $-s\ell>0$. Set $C=[A,B]$. Working modulo $Z(H)$ and taking
$i=r$ and $j=-s\ell$ in Lemma \ref{basw}, we obtain
\begin{equation}
\label{pok}
[A^r,B^{s\ell}]=A^{-(\alpha-1)\ell(r-1)/2} C^{\ell}B^{-(\alpha-1)\ell(j+1)/2}C^{kp^m},\quad k\in\Z.
\end{equation}
On the other hand, since $\beta\equiv\alpha^\ell\mod p^{2m}$, we have
\begin{equation}
\label{tel}
(A^r)^{C^\ell}=(A^r)^\beta,\; (B^{s\ell})^{C^{-\ell}}=(B^{s\ell})^\beta.
\end{equation}
The centralizers of $A$ and $B$ in $H$ were computed in the proof of \cite[Theorem 8.1]{MS}.
Thus, combining (\ref{tel0}) and (\ref{tel}), we deduce
$$
[A^r,B^{s\ell}]C^{-\ell}\in C_{H}(A^r)=C_{H}(A)=\langle A,C^{p^m}\rangle,
$$
$$
[B^{s\ell},A^r]C^{\ell}\in C_{H}(B^{s\ell})=C_{H}(B)=\langle B,C^{p^m}\rangle.
$$
These last two equations and (\ref{pok}) yield
$$
B^{-\alpha^\ell(\alpha-1)\ell(j+1)/2}=C^{\ell}B^{-(\alpha-1)\ell(j+1)/2}C^{-\ell}\in \langle A,C^{p^m}\rangle,
$$
$$
A^{\alpha^\ell(\alpha-1)\ell(r-1)/2}=C^{-\ell}A^{(\alpha-1)\ell(r-1)/2}C^{\ell}\in \langle B,C^{p^m}\rangle.
$$
The normal form of the elements of $H$ now forces $j\equiv -1\mod p^m$ and $r\equiv 1\mod p^m$. The latter
yields $s\equiv 1\mod p^m$ and since $j=-s\ell$, we deduce $\ell\equiv 1\mod p^m$.
\end{proof}

Recall that $K=H/Z(H)$, with presentation
$$
K=\langle u,v\,|\, u^{[u,v]}=u^\alpha, v^{[v,u]}=v^\alpha, u^{p^{2m}}=1, v^{p^{2m}}=1, [u,v]^{p^{m}}=1\rangle,
$$
and an automorphism $\mu$, defined by $u\leftrightarrow v$.
We have a group epimorphism $\pi:H\to K$ given by $A\mapsto u$, $B\mapsto v$. Since
$Z(H)$ is a characteristic subgroup of $H$, we see that $\pi$ gives
rise to a group homomorphism $\Lambda:\mathrm{Aut}(H)\to \mathrm{Aut}(K)$ such that $\pi g^\Lambda= g\pi$
for all $g\in \mathrm{Aut}(H)$.

\begin{prop}\label{tet} The map
$\Lambda:\mathrm{Aut}(H)\to \mathrm{Aut}(K)$ sends $\mathrm{Aut}_3(H)\langle\nu\rangle$ onto
$\mathrm{Aut}_2(K)\langle\mu\rangle$.
\end{prop}

\begin{proof} By Proposition \ref{auth}, $\Lambda$ maps
$\mathrm{Aut}_2(H)$ onto
$\mathrm{Aut}_1(K)$, and it is clear that $\Lambda$ sends $\mathrm{Inn}(H)$ onto $\mathrm{Inn}(K)$.
Thus $\Lambda$ maps $\mathrm{Aut}_3(H)$ onto $\mathrm{Aut}_2(K)$. Obviously, $\Lambda$ sends $\nu$ into $\mu$,
so $\Lambda$ maps $\mathrm{Aut}_3(H)\langle\nu\rangle$ onto $\mathrm{Aut}_2(K)\langle\mu\rangle$.
\end{proof}

\begin{theorem}\label{auth6} We have $\mathrm{Aut}(H)=\mathrm{Aut}_3(H)\langle \nu\rangle$. In particular,
$|\mathrm{Aut}(H)|=2p^{8m}$.
\end{theorem}

\begin{proof}
By definition, $\ker(\Lambda)=\mathrm{Aut}_1(H)$. On the other hand,
$\Lambda$ maps $\mathrm{Aut}_3(H)\langle\nu\rangle$ onto $\mathrm{Aut}_2(K)\langle\mu\rangle$ by Proposition \ref{tet}, .
Since $\ker(\Lambda)$ is included in $\mathrm{Aut}_3(H)$, it follows
that $\mathrm{Aut}(H)=\mathrm{Aut}_3(H)\langle \nu\rangle$
if and only if $\mathrm{Im}(\Lambda)=\mathrm{Aut}_2(K)\langle\mu\rangle$. Let $S$ be the subgroup of $\mathrm{Aut}_2(K)$
defined Section \ref{autgmod2}, so that $\mathrm{Aut}(K)=\mathrm{Aut}_2(K)S\langle\mu\rangle$ by Theorem \ref{autk6}.
It follows that $\mathrm{Im}(\Lambda)=\mathrm{Aut}_2(K)\langle\mu\rangle$
if and only if the only elements $f_r$ of $S$ in $\mathrm{Im}(\Lambda)$
are those from $(\mathrm{Aut}_2(K)\langle\mu\rangle)\cap S$. According to Proposition \ref{autk4},
we have $f_r\in \mathrm{Aut}_2(K)\cap S$ if and only if $r\equiv 1\mod p^m$, and the proof of Theorem \ref{autk6}
makes it clear that $(\mathrm{Aut}_2(K)S)\cap\langle\mu\rangle$ is trivial, so
$(\mathrm{Aut}_2(K)\langle\mu\rangle)\cap S=\mathrm{Aut}_2(K)\cap S$.

Thus, we are reduced to show that $f_r\in \mathrm{Im}(\Lambda)$ only when $r\equiv 1\mod p^m$.
Suppose then that $f_r=g^\Lambda$ for some $g\in \mathrm{Aut}(H)$. From $\pi g^\Lambda= g\pi$ we deduce
$A^r\pi=Ag\pi$ and $B^s\pi=Bg\pi$, where $rs\equiv 1\mod p^{2m}$. Since $\ker(\pi)=Z(H)$, there exist
$z_1,z_2\in Z(H)$ such that $Ag=A^r z_1$, $Bg=A^s z_2$. As $g\in \mathrm{Aut}(H)$ the relations
$A^{[A,B]}=A^\alpha, B^{[B,A]}=B^\alpha$ must be preserved, so (\ref{pli}) holds with $\ell=1$ and $\beta=\alpha$.
It follows from Proposition \ref{tecnico} that $r\equiv 1\mod p^m$, as required.
\end{proof}

In terms of group extensions, we have the normal series
$$
1\subset \mathrm{Aut}_1(H)\subset\mathrm{Aut}_2(H)\subset \mathrm{Aut}_3(H)\subset \mathrm{Aut}(H).
$$
The first factor $\mathrm{Aut}_1(H)\cong (\Z/p^m\Z)^2$ by Proposition \ref{auth}. The second factor
$\mathrm{Aut}_2(H)/\mathrm{Aut}_1(H)\cong Z_2(H)^2/Z(H)^2\cong (\Z/p^m\Z)^4$ also by Proposition \ref{auth},
which yields $\mathrm{Aut}_2(H)\cong (\Z/p^m\Z)^6$ directly as well. The third factor
$$
\mathrm{Aut}_3(H)/\mathrm{Aut}_2(H)\cong \mathrm{Aut}_2(H)\mathrm{Inn}(H)/\mathrm{Aut}_2(H)\cong
\mathrm{Inn}(H)/\mathrm{Inn}(H)\cap \mathrm{Aut}_2(H)\cong H\delta/Z_3(H)\delta
$$
is isomorphic to $(H/Z(H))/(Z_3(H)/Z(H))\cong H/Z_3(H)\cong (\Z/p^m\Z)^2$. The fourth factor
$$
\mathrm{Aut}(H)/\mathrm{Aut}_3(H)\cong \Z/2\Z
$$
by Theorem \ref{auth6}.

We next find a presentation for the normal Sylow $p$-subgroup  $\mathrm{Aut}_3(H)$
of $\mathrm{Aut}(H)$ of order~$p^{8m}$. We appeal to the isomorphism
$\Pi:Z_2(H)\times Z_2(H)\to \mathrm{Aut}_2(H)$ defined in Proposition \ref{auth}.

Recall that $\mathrm{Aut}_3(H)=\mathrm{Aut}_2(H)\mathrm{Inn}(H)$. We have
$\mathrm{Aut}_2(H)\cap\mathrm{Inn}(H)=Z_3(H)\delta$. Now
$A^{p^m}\delta$ fixes $A$ and sends $B$ to $BC^{-p^m}$, so $A^{p^m}\delta=(1,C^{-p^m})\Pi$.
Moreover, $B^{p^m}\delta$ fixes $B$ and sends $A$ to $AC^{p^m}$, whence $A^{p^m}\delta=(C^{p^m},1)\Pi$.
Furthermore, $C\delta$ sends $A$ to $A^{\alpha}$ and $B$ to $B^{2-\alpha}$, hence
$C\delta=(A^{p^m k}, B^{-p^m k})\Pi$, where $\alpha-1=p^m k$. Since $Z_2(H)\times Z_2(H)$
is the internal direct product of the cyclic subgroups generated by
$$
(A^{p^m},1), (B^{p^m},1), (1,A^{p^m}), (C^{p^m},1), (1, C^{-p^m}), (A^{p^m k}, B^{-p^m k}),
$$
it follows that
$$
\mathrm{Aut}_3(H)=\langle (A^{p^m},1)\Pi, (B^{p^m},1)\Pi, (1,A^{p^m})\Pi\rangle\ltimes \mathrm{Inn}(H).
$$
Here $\mathrm{Inn}(H)\cong K$ is generated by
\begin{equation}\label{yu1}
a=A\delta, b=B\delta
\end{equation}
subject to the defining relations
\begin{equation}\label{relt1}
a^{[a,b]}=a^\alpha, b^{[b,a]}=b^\alpha, a^{p^{2m}}=1, b^{p^{2m}}=1, [a,b]^{p^{m}}=1.
\end{equation}
Set
\begin{equation}\label{yu2}
x=(A^{p^m},1)\Pi,\; y=(B^{p^m},1)\Pi,\; z=(1,A^{p^m})\Pi.
\end{equation}
Then $x,y,z$ generate a subgroup of $\mathrm{Aut}(H)$ isomorphic to $(\Z/p^m\Z)^3$, with defining relations
\begin{equation}\label{relt2}
x^{p^m}=y^{p^m}=z^{p^m}=1, [x,y]=[x,z]=[y,z]=1.
\end{equation}
We next use the formula $(u\delta)^T=(u^T)\delta$ for $T\in\mathrm{Aut}(H)$ and $u\in H$, applied to
$T\in\{x,y,z\}$ and $u\in \{A,B\}$, and obtain
\begin{equation}\label{relt3}
a^x=a^{1+p^m}, b^x=b, b^y=b^{1+p^m}, a^y=a, a^z=a, b^z=ba^{p^m}.
\end{equation}

We have proven the following result.

\begin{theorem} The group $\mathrm{Aut}_3(H)$ is a normal
Sylow $p$-subgroup of $\mathrm{Aut}(H)$ of order $p^{8m}$ and index 2 in $\mathrm{Aut}(H)$. It is generated by
$a,b,x,y,z$, as defined in (\ref{yu1}) and (\ref{yu2}), subject to the defining relations (\ref{relt1}), (\ref{relt2}),
and (\ref{relt3}).
\end{theorem}

\section{The automorphism group of $J$}\label{sejotar}

According to \cite[Theorem 8.1]{MS}, $J$ is nilpotent of class 5, while \cite[Theorem 5.3]{MS} ensures that $J$ has presentation
$$
J=\langle A,B\,|\, A^{[A,B]}=A^\alpha, B^{[B,A]}=B^\alpha, A^{p^{3m}}=1=B^{p^{3m}}\rangle,
$$
and we set $C=[A,B]$. We have an automorphism $A\leftrightarrow B$, say $\theta$, of $J$.

By \cite[Theorem 7.1]{MS}, $|J|=p^{7m}$, $A,B,C$ have respective orders $p^{3m},p^{3m},p^{2m}$,
and every element of $J$ can be written in one and only one way in the form
$A^iB^jC^k$, where $0\leq i<p^{3m}$ and $0\leq j,k<p^{2m}$. Here $A^{p^{2m}}B^{p^{2m}}=1$, as indicated
in \cite[Section 6]{MS}.

Due to \cite[Theorem 8.1]{MS}, we have
$$
Z(J)=\langle A^{p^{2m}}\rangle,
Z_2(J)=\langle A^{p^{2m}}, C^{p^{m}}\rangle,\;
Z_3(J)=\langle A^{p^{m}}, B^{p^{m}}, C^{p^{m}}\rangle,\; Z_4(J)=\langle A^{p^{m}}, B^{p^{m}}, C\rangle.
$$
Here $Z_3(J)$ is abelian of order $p^{4m}$ by \cite[Proposition 8.1]{MS}. It follows that $Z(J)\cong\Z/p^m\Z$
and $Z_2(J)\cong(\Z/p^m\Z)^2$. The isomorphism $J/Z_3(J)\cong \mathrm{Heis}(\Z/p^m\Z)$ yields $J/Z_4(J)\cong (\Z/p^m\Z)^2$,
whence $|Z_4(J)|=p^{5m}$. We also have $\mathrm{Inn}(J)\cong J/Z(J)\cong H$, a group of order $p^{6m}$, where $H$ is defined in Section \ref{autgmod56}.

\begin{prop}\label{autg} For every $x,y\in Z_2(J)$ the assignment
$$
A\mapsto Ax,\; B\to By
$$
extends to a 2-central automorphism $(x,y)\Psi$ of $J$ that fixes $Z_3(J)$ pointwise. Moreover, the corresponding map
$\Psi:Z_2(J)\times Z_2(J)\to \mathrm{Aut}(J)$ is a group monomorphism whose image is $\mathrm{Aut}_2(J)$. In particular,
$|\mathrm{Aut}_2(J)|=p^{4m}$.
\end{prop}

\begin{proof} We begin by showing that
\begin{equation}
\label{popy}
(Au)^\alpha=A^\alpha u,\; (Bu)^{\alpha}=B^\alpha u,\; (Au)^{p^{3m}}=1=(Bu)^{p^{3m}}\quad u\in Z_2(J).
\end{equation}

Let $u\in Z_2(J)$. Then $u=A^{j p^{2m}}C^{\ell p^{m}}$, with $j,\ell\in\Z$, $A^{p^{2m}}\in Z(J)$ and $[C,u]=1$.
Then
$$
(Au)^{p^{m}}=(A A^{j p^{2m}}C^{\ell p^{m}})^{p^{m}}=
(A C^{p^{m}\ell})^{p^{m}}.
$$
Set $\beta=\alpha^{\ell p^{m}}$. Then
$$
(Au)^{p^{m}}=C^{\ell p^{2m}}A^{\beta(\beta^{p^{m}}-1)/(\beta-1)}.
$$
Here $\alpha=1+kp^m$, with $p\nmid k$, and we deduce
\begin{equation}
\label{bcer}
\beta\equiv 1+k\ell p^{2m}\mod p^{3m}.
\end{equation}
On the other hand,
$$
\beta^{p^{m}}=(1+(\beta-1))^{p^{m}}=1+p^{m}(\beta-1)+{{p^{m}}\choose{2}}(\beta-1)^2+{{p^{m}}\choose{3}}
(\beta-1)^3+\cdots.
$$
Since $\beta-1\equiv 0\mod p^{2m}$ by (\ref{bcer}), we see that
\begin{equation}
\label{bcer2}
(\beta^{p^{m}}-1)/(\beta-1)\equiv p^{m}\mod  p^{3m},
\end{equation}
so (\ref{bcer}) and (\ref{bcer2}) yield
$$
\beta(\beta^{p^{m}}-1)/(\beta-1)\equiv (1+k\ell p^{2m})p^{m}\equiv p^{m}\mod  p^{3m}.
$$
This proves that $(Au)^{p^{m}}=A^{p^{m}}$, and applying the automorphism $A\leftrightarrow B$ of $J$ we obtain
\begin{equation}
\label{mat1}
(Au)^{p^{m}}=A^{p^{m}},\; (Bu)^{p^{m}}=B^{p^{m}}.
\end{equation}
It follows from (\ref{mat1}) that  $(Au)^{p^{3m}}=1$ and $(Bu)^{p^{3m}}=1$.
Moreover, since $\alpha=1+r p^m$, with $r\in\Z$, we deduce $[A^{p^m},C^{p^m}]=1$, hence
$[A^{p^m},u]=1$, and therefore (\ref{mat1}) gives $(Au)^\alpha=Au (Au)^{\alpha-1}=Au A^{rp^m}=A^\alpha u$,
whence the automorphism $A\leftrightarrow B$ of $J$ yields $(Bu)^\alpha=B^\alpha u$.
This proves (\ref{popy}). On the other hand, working modulo $Z(J)$, we see that (\ref{comfor}) implies
\begin{equation}
\label{mat6}
[Ax,Ay]=[A,B]z
\end{equation}
for a unique $z\in Z(J)$. Since $[C,x]=1=[C,y]$, (\ref{popy}) and (\ref{mat6}) yield
$$
(Ax)^{[Ax,By]}=(Ax)^{[A,B]}=A^\alpha x=(Ax)^\alpha, (By)^{[By,Bx]}=(By)^{[B,A]}=B^\alpha y=(By)^\alpha,A^{p^{3m}}=1=B^{p^{3m}}.
$$
This shows that the defining relations of $J$ are preserved, which
yields an endomorphism $(x,y)\Psi$ of $J$ that fixes $A^{p^m}$ and $B^{p^m}$ by (\ref{mat1}).
Moreover,
$$
C^{p^{m}}\mapsto (Cz)^{p^{m}}=C^{p^{m}}z^{p^{m}}=C^{p^{m}},
$$
because $z\in\langle Z(J)\rangle$, a group of exponent $p^{m}$. Thus $(x,y)\Psi$ fixes $Z_3(J)$ pointwise.
We may now continue the argument used in the proof of Proposition \ref{autk}.
\end{proof}

\begin{prop}\label{autg2} We have $\mathrm{Inn}(J)\cap\mathrm{Aut}_2(J)=
\langle C^{\alpha-1}\delta, A^{\alpha-1}\delta, B^{\alpha-1}\delta\rangle$, $|\mathrm{Inn}(J)\cap\mathrm{Aut}_2(J)|=p^{3m}$,
and $|\mathrm{Inn}(J)\mathrm{Aut}_2(J)|=p^{7m}$. In fact,
$\mathrm{Inn}(J)\mathrm{Aut}_2(J)=\langle (A^{p^{2m}},A^{-p^{2m}})\Psi\rangle\ltimes\mathrm{Inn}(J)$.
\end{prop}

\begin{proof} As $Z_3(J)/Z(J)$ has order $p^{3m}$ and is generated by the cosets of $C^{\alpha-1},A^{\alpha-1}, B^{\alpha-1}$,
it follows that
$\mathrm{Inn}(J)\cap\mathrm{Aut}_2(J)=\langle C^{\alpha-1}\delta,A^{\alpha-1}\delta, B^{\alpha-1}\delta\rangle$ has order $p^{3m}$.
Thus $\mathrm{Inn}(J)\mathrm{Aut}_2(J)$  has order $p^{6m}p^{4m} p^{-3m}=p^{7m}$. We claim
that $\mathrm{Inn}(J)\cap\langle (A^{p^m},A^{-p^{m}})\Psi\rangle$ is trivial, which readily yields
the last assertion. To see the claim, suppose $(A^{rp^{2m}},A^{-rp^{2m}})\Psi=x\delta$ for some $r\in\Z$ and $x\in J$.
Then $A^x=A^{1+rp^{2m}}$ and $B^x=BA^{-r p^{2m}}=B^{1+rp^{2m}}$, which
implies $x\in N_J(\langle A\rangle)\cap N_J(\langle B \rangle)$. It follows from \cite[Propositions 7.2 and 7.3]{MS} that 
$N_J(\langle A \rangle)\cap N_J(\langle B \rangle)=
\langle A^{p^{2m}},C\rangle$, so $x=A^{sp^{2m}}C^t$ for some $s,t\in\N$, whence $A^{1+rp^{2m}}=A^x=A^{C^t}=A^{\alpha^t}$.
Here $\alpha=1+kp^m$, with $k\in\Z$ and $\gcd(k,p)=1$, so we deduce
$t=sp^m$ for some $s\in\Z$ such that $sk\equiv r\mod p^m$.
As $(1+rp^{2m})(1-rp^{2m})\equiv 1\mod p^{3m}$, we infer
$B^x=B^{C^t}=B^{1-rp^{2m}}$. But $B^x=B^{1+rp^{2m}}$ as well, so $r\equiv 0\mod p^m$, as claimed.
\end{proof}

\begin{prop}\label{zi2} Let $T$ be a group, and set $Y=T/Z(T)$. Let $\lambda:T\to Y$ be the canonical projection,
and consider the associated map $\Lambda:\mathrm{Aut}(T)\to \mathrm{Aut}(Y)$.
Then for any $i\geq 0$, $\Lambda$ maps $\mathrm{Aut}_{i+2}(T)$ into $\mathrm{Aut}_{i+1}(Y)$, and the kernels of the
induced maps $\mathrm{Aut}_{i+2}(T)\to\mathrm{Aut}_{i+1}(Y)/\mathrm{Aut}_{i}(Y)$ and
$\mathrm{Inn}(T)\mathrm{Aut}_{i+2}(T)\to\mathrm{Inn}(Y)\mathrm{Aut}_{i+1}(Y)/\mathrm{Inn}(Y)\mathrm{Aut}_{i}(Y)$ are
$\mathrm{Aut}_{i+1}(T)$ and $\mathrm{Inn}(T)\mathrm{Aut}_{i+1}(T)$, respectively.
Thus, $\mathrm{Aut}_{i+2}(T)/\mathrm{Aut}_{i+1}(T)$ is isomorphic to a subgroup of
$\mathrm{Aut}_{i+1}(Y)/\mathrm{Aut}_{i}(Y)$,
and $\mathrm{Inn}(T)\mathrm{Aut}_{i+2}(T)/\mathrm{Inn}(T)\mathrm{Aut}_{i+1}(T)$ imbeds into
$\mathrm{Inn}(Y)\mathrm{Aut}_{i+1}(Y)/\mathrm{Inn}(Y)\mathrm{Aut}_{i}(Y)$.
\end{prop}

\begin{proof} We first show by induction that $\lambda$ sends $Z_{i+1}(T)$ onto $Z_{i}(Y)$ for any $i\geq 0$.
The base case $i=0$ holds by the very definition of $\lambda$. Suppose that $\lambda$ maps $Z_{i+1}(T)$ onto $Z_i(Y)$
for some $i\geq 0$. Let $t\in Z_{i+2}(T)$. Then $[t,\sigma]\in Z_{i+1}(T)$, so
$[t^\lambda,\sigma^\lambda]=[t,\sigma]^\lambda\in Z_i(Y)$ for every $\sigma\in T$,
whence $t^\lambda\in Z_{i+1}(Y)$. Conversely, if $y\in Z_{i+1}(Y)$, then $y=t^\lambda$ for some $t\in T$. Let $\sigma\in T$.
Then $[t,\sigma]^\lambda=[t^\lambda,\sigma^\lambda]\in Z_i(Y)$, so $[t,\sigma]^\lambda=r^\lambda$ for some $r\in Z_{i+1}(T)$,
hence $[t,\sigma]r^{-1}\in\ker(\lambda)=Z(T)$, and therefore $[t,\sigma]\in Z(T)Z_{i+1}(T)=Z_{i+1}(T)$, which implies $t\in Z_{i+2}(T)$.
This shows that $\lambda$ maps $Z_{i+2}(T)$ onto $Z_i(Y)$.

For $t\in T$, we set $\overline{t}=t^\lambda$, and if $g\in \mathrm{Aut}(T)$, then $\overline{g}=g^\Lambda$ is defined
by $\overline{t}^{\overline{g}}=\overline{t^g}$. This is well-defined, as $Z(T)$ is a characteristic subgroup of $T$.

Let $i\geq 0$. We claim that $\Lambda$ sends $\mathrm{Aut}_{i+2}(T)$ into $\mathrm{Aut}_{i+1}(Y)$.
Indeed, let $g\in \mathrm{Aut}_{i+2}(T)$ and
$t\in T$. Then $t^g t^{-1}\in Z_{i+2}(T)$, so $\overline{t^g t^{-1}}\in Z_{i+1}(Y)$ by above, which means
$\overline{t}^{\overline{g}} \overline{t}^{-1} Z_{i+1}(Y)$, so $g^\Lambda\in \mathrm{Aut}_{i+1}(Y)$,  as claimed.

By above, we have a group homomorphism $\eta:\mathrm{Aut}_{i+2}(T)\to \mathrm{Aut}_{i+1}(Y)/\mathrm{Aut}_{i}(Y)$ with
$\mathrm{Aut}_{i+1}(T)$ contained in $\ker(\eta)$, and we claim equality prevails.
Let $g\in\ker(\eta)$. Then $g^\Lambda\in \mathrm{Aut}_{i}(Y)$. Let $t\in T$.
Then $\overline{t}^{\overline{g}} \overline{t}^{-1}\in Z_i(Y)$, so $\overline{t^g t^{-1}}\in Z_{i}(Y)$. As $\lambda$
maps $Z_{i+1}(T)$ onto $Z_i(Y)$, there is some $\sigma\in Z_{i+1}(T)$ such that $\overline{t^g t^{-1}}=\overline{\sigma}$,
whence $t^g t^{-1}\sigma^{-1}\in Z(T)$, so $t^g t^{-1}\in Z(T)Z_{i+1}(T)=Z_{i+1}(T)$, and therefore $g\in \mathrm{Aut}_{i+1}(T)$.
Thus $\ker(\eta)=\mathrm{Aut}_{i+1}(T)$, as claimed.

Now $\Lambda$ maps $\mathrm{Inn}(T)\mathrm{Aut}_{i+2}(T)$ into
$\mathrm{Inn}(Y)\mathrm{Aut}_{i+1}(Y)$, and $\mathrm{Inn}(T)\mathrm{Aut}_{i+1}(T)$ into
$\mathrm{Inn}(Y)\mathrm{Aut}_{i}(Y)$, thus producing a map $\phi:\mathrm{Inn}(T)\mathrm{Aut}_{i+2}(T)\to
\mathrm{Inn}(Y)\mathrm{Aut}_{i+1}(Y)/\mathrm{Inn}(Y)\mathrm{Aut}_{i}(Y)$ whose kernel contains $\mathrm{Inn}(T)\mathrm{Aut}_{i+1}(T)$,
and we claim that $\ker(\phi)=\mathrm{Inn}(T)\mathrm{Aut}_{i+1}(T)$. Indeed, if
$g\in \ker(\phi)$, then $g^\Lambda\in \mathrm{Inn}(Y)\mathrm{Aut}_{i}(Y)$, so that $g^\Lambda=hk$,
with $h\in \mathrm{Inn}(Y)$ and $k\in\mathrm{Aut}_{i}(Y)$. But $\Lambda$ maps $\mathrm{Inn}(T)$ onto $\mathrm{Inn}(Y)$,
so $h=w^\Lambda$ for some $w\in \mathrm{Inn}(T)$, whence $k=(w^{-1}g)^\Lambda$, so by above $w^{-1}g\in \mathrm{Aut}_{i+1}(T)$
and therefore $h\in \mathrm{Inn}(T)\mathrm{Aut}_{i+1}(T)$, as claimed.
\end{proof}

We write $\alpha=1+\ell p^m$, with $\ell\in\N$, for the remainder of this section.

\begin{theorem}\label{autj3} The assignment $A\mapsto AB^{p^m}$, $B\mapsto B B^{-(2-d)p^m}A^{p^m}$, where $d=0$ if $p\neq 3$
and $d=3^{m-1}\ell$ if $p=3$, extends to an automorphism, say $\Delta$, of $J$. Moreover,
$$\mathrm{Inn}(J)\mathrm{Aut}_3(J)=\langle \Delta\rangle \mathrm{Inn}(J)\mathrm{Aut}_2(J),\;
\mathrm{Aut}_3(J)=\langle \Delta, C\delta\rangle\mathrm{Aut}_2(J),
$$
$$\mathrm{Aut}_3(J)/\mathrm{Aut}_2(J)\cong\Z/p^m\Z\times \Z/p^m\Z,\;
\mathrm{Inn}(J)\mathrm{Aut}_3(J)/\mathrm{Inn}(J)\mathrm{Aut}_2(J)\cong\Z/p^m\Z,
$$
$$
|\mathrm{Aut}_3(J)|=p^{6m},\; |\mathrm{Inn}(J)\mathrm{Aut}_3(J)|=p^{8m}.
$$
\end{theorem}

\begin{proof} Applying Proposition \ref{zi2} twice, we obtain
$$
\mathrm{Aut}_3(J)/\mathrm{Aut}_2(J)\hookrightarrow
\mathrm{Aut}_2(H)/\mathrm{Aut}_1(H)\hookrightarrow
\mathrm{Aut}_1(K),
$$
where $\mathrm{Aut}_1(K)\cong Z(K)\times Z(K)$ by Proposition \ref{autk}. In this way we obtain a homomorphism
$\psi:\mathrm{Aut}_3(J)\to Z(K)^2$ such that $\ker(\psi)=\mathrm{Aut}_2(J)$. We proceed to show that
\begin{equation}
\label{impsi}
\mathrm{Im}(\psi)=\langle (a^{p^{m}},b^{-p^{m}}),(b^{p^{m}},b^{-(2-d)p^{m}} a^{p^{m}})\rangle,
\end{equation}
where $\rho: J\to K$ is the canonical projection, $a=A\rho$, $b=B\rho$, and $c=C\rho$.

Let $i,j,k,l\in\N$ and $w,z\in Z_2(J)$. We need to determine under what conditions the assignment
$$
A\mapsto AA^{i p^m}B^{j p^m}w,\; B\mapsto B A^{k p^m} B^{l p^m}z
$$
extends to an automorphism of $J$. By Proposition \ref{autg}, this is the case if and only if
\begin{equation}
\label{pupicat}
A\mapsto AA^{i p^m}B^{j p^m},\; B\mapsto B A^{k p^m} B^{l p^m}
\end{equation}
extends to an automorphism of $J$. We will implicitly use in what follows that $Z_3(J)$ is abelian and that
$A^{p^{2m}},B^{p^{2m}}\in Z(J)$, as indicated above,

We easily see that
\begin{equation}
\label{commie2}
B^{A^s}=BA^{(\alpha-1)s(s-1)/2}C^{-s},\quad s\geq 1.
\end{equation}

For a finite abelian $p$-group, we write $x\mapsto x^{1/2}$ for the inverse automorphism of $x\mapsto x^2$.
Using this notation, and taking $s=r p^m$ in (\ref{commie2}), we obtain
\begin{equation}
\label{commie3}
B^{A^{r p^m}}=BA^{r p^m(1-\alpha)/2}C^{-r p^m},\quad r\in\N.
\end{equation}
which implies $(B^{-1})^{A^{r p^m}}=C^{r p^m}A^{r p^m(\alpha-1)/2}B^{-1}$, whence
\begin{equation}
\label{commie4}
[A^{r p^m},B]=C^{r p^m}A^{r p^m(\alpha-1)/2},\quad r\in\N.
\end{equation}
Applying the automorphism $A\leftrightarrow B$ to (\ref{commie4}) yields
\begin{equation}
\label{commie5}
[B^{r p^m},A]=C^{-r p^m}B^{r p^m(\alpha-1)/2}, \quad r\in\N.
\end{equation}
From $C^A=C A^{1-\alpha}$ and $C^B=B^{\alpha-1}C$, we derive
\begin{equation}
\label{commie6}
C^{A^{r p^m}}=C A^{(1-\alpha)r p^m},\; C^{B^{r p^m}}=B^{(\alpha-1)r p^m}C.
\end{equation}

On the other hand, appealing to (\ref{comfor}), we see that
\begin{equation}
\label{prito}
[AA^{i p^m}B^{j p^m},B A^{k p^m} B^{l p^m}]=
[AA^{i p^m}B^{j p^m},A^{k p^m} B^{l p^m}][AB^{i p^m}B^{j p^m},B]^{A^{k p^m} B^{l p^m}}.
\end{equation}
A second application of (\ref{comfor}) gives
\begin{equation}
\label{prito2}
[AA^{i p^m}B^{j p^m},A^{k p^m} B^{l p^m}]=[A,A^{k p^m} B^{l p^m}]^{A^{i p^m}B^{j p^m}},
\end{equation}
and a third use of (\ref{comfor}) yields
\begin{equation}
\label{prito3}
[A,A^{k p^m} B^{l p^m}]=[A,B^{l p^m}]=[B^{l p^m},A]^{-1}.
\end{equation}
Resorting to (\ref{comfor}) a fourth time, we obtain
\begin{equation}
\label{prito4}
[AA^{i p^m}B^{j p^m},B]=[A,B]^{A^{i p^m}B^{j p^m}}[A^{i p^m}B^{j p^m},B],
\end{equation}
while a final usage of (\ref{comfor}) gives
\begin{equation}
\label{prito5}
[A^{i p^m}B^{j p^m},B]=[A^{i p^m},B]^{B^{j p^m}}.
\end{equation}
Combining (\ref{commie4})-(\ref{prito5}), we derive
\begin{equation}
\label{tuto}
[A A^{i p^m}B^{j p^m},B B^{l p^m}A^{k p^m}]= C^{1+(i+l)p^m}A^{(2k+i)p^m(1-\alpha)/2}
B^{(2j+l)p^m(\alpha-1)/2}.
\end{equation}
It follows from (\ref{tuto}) that
\begin{equation}
\label{tuto2}
(A A^{i p^m}B^{j p^m})^{[A A^{i p^m}B^{j p^m},B B^{l p^m}A^{k p^m}]}=A^{C^{1+(i+l)p^m}} (A^{i p^m})^C (B^{j p^m})^C.
\end{equation}
To compute with (\ref{tuto2}) recall that $\alpha=1+\ell p^m$, which
implies $\alpha^{i p^m}\equiv 1+i\ell p^{2m}\mod p^{3m}$ and $\alpha^{l p^m}\equiv 1+l \ell p^{2m}\mod p^{3m}$. Moreover,
we readily see that $(1+\ell p^m)(1-\ell p^m+\ell^2 p^{2m})\equiv 1\mod p^{3m}$, so that $B^C=B^{1-\ell p^m+\ell^2 p^{2m}}$.
We may now appeal to (\ref{tuto2}) and the relation $A^{p^{2m}}B^{p^{2m}}=1$ to obtain
\begin{equation}
\label{tuto3}
(A A^{i p^m}B^{j p^m})^{[A A^{i p^m}B^{j p^m},B B^{l p^m}A^{k p^m}]}=A^{1+(\ell+i)p^m+(2i+l+j)\ell p^{2m}}B^{jp^m}.
\end{equation}
Making the substitutions $i\leftrightarrow l$ and $j\leftrightarrow k$ in (\ref{tuto3}) gives
\begin{equation}
\label{tuto4}
(A A^{l p^m}B^{k p^m})^{[A A^{l p^m}B^{k p^m},B B^{i p^m}A^{j p^m}]}=A^{1+(\ell+l)p^m+(2l+i+k)\ell p^{2m}}B^{kp^m}.
\end{equation}
Applying the automorphism $A\leftrightarrow B$ to (\ref{tuto4}) yields
\begin{equation}
\label{tuto5}
(B B^{l p^m}A^{k p^m})^{[B B^{l p^m}A^{k p^m},A A^{i p^m}B^{j p^m}]}=B^{1+(\ell+l)p^m+(2l+i+k)\ell p^{2m}}A^{kp^m}.
\end{equation}

We wish to compare $(A A^{i p^m}B^{j p^m})^\alpha$ with (\ref{tuto3}), and $(B B^{l p^m}A^{k p^m})^\alpha$
with (\ref{tuto5}). We begin by computing $(A A^{i p^m}B^{j p^m})^{p^m}$. We first note that
\begin{equation}
\label{pip}
(A A^{i p^m}B^{j p^m})^{p^m}=A^{i p^{2m}}(A B^{j p^m})^{p^m}.
\end{equation}
Next observe that
\begin{equation}
\label{pip2}
(A B^{j p^m})^{p^m}=B^{j p^{2m}}A^{B^{p^m jp^m}}\cdots A^{B^{2jp^m}}A^{B^{jp^m}}.
\end{equation}
Applying the automorphism $A\leftrightarrow B$ to (\ref{commie3}) with $r=uj$, we have
\begin{equation}
\label{pip3}
A^{B^{ujp^m}}=AB^{uj p^m(1-\alpha)/2}C^{uj p^m},\quad u\in\N.
\end{equation}
Substituting (\ref{pip3}), with $u\in\{1,\dots,p^m\}$, in the rightmost factors of (\ref{pip2}) yields
\begin{equation}
\label{pip4}
A^{B^{p^m jp^m}}\cdots A^{B^{2jp^m}}A^{B^{jp^m}}=B^{(1+2+\cdots+p^m)j p^m(1-\alpha)/2}
A C^{p^m j p^m}A C^{(p^m-1)j p^m}\cdots A  C^{2j p^m}A  C^{j p^m}.
\end{equation}
Now $B^{(1+2+\cdots+p^m)j p^m(1-\alpha)/2}=1=C^{p^m j p^m}$. Thus, (\ref{pip4}) gives
\begin{equation}
\label{pip5}
A^{B^{p^m jp^m}}\cdots A^{B^{2jp^m}}A^{B^{jp^m}}=
A C^{jp^m {{p^m}\choose{2}}} A^{C^{jp^m {{p^m}\choose{2}}}}\cdots A^{C^{jp^m {{3}\choose{2}}}} A^{C^{jp^m {{2}\choose{2}}}}.
\end{equation}
Here $C^{jp^m {{p^m}\choose{2}}}=1$ and $A^{C^{jp^m {{u}\choose{2}}}}=A^{1+\ell j p^{2m} {{u}\choose{2}}}$, $u\in\{2,\dots,p^m\}$,
so (\ref{pip}), (\ref{pip2}), and (\ref{pip5}) yield
\begin{equation}
\label{pip6}
(A A^{i p^m}B^{j p^m})^{p^m}=A^{i p^{2m}} B^{j p^{2m}} A^{p^m} A^{\ell j p^{2m}({{2}\choose{2}}+{{3}\choose{2}}+\cdots+{{p^m}\choose{2}})}.
\end{equation}
Thus, recalling the meaning of $d$, we infer from (\ref{freezing2}) and (\ref{pip6}) that
\begin{equation}
\label{pip7}
(A A^{i p^m}B^{j p^m})^{p^m}=A^{p^m+(i+jd)p^{2m}} B^{j p^{2m}}.
\end{equation}
It now follows easily from (\ref{pip7}), $\alpha=1+\ell p^m$, and the relation $A^{p^{2m}}B^{p^{2m}}=1$, that
\begin{equation}
\label{pip8}
(A A^{i p^m}B^{j p^m})^{\alpha}=A^{1+(i+\ell)p^m+\ell(i+jd-j)p^{2m}} B^{j p^{m}}.
\end{equation}
Making the substitutions $i\leftrightarrow l$ and $j\leftrightarrow k$ in (\ref{pip7}) and (\ref{pip8}) gives
\begin{equation}
\label{zim}
(A A^{l p^m}B^{k p^m})^{p^m}=A^{p^m+(l+kd)p^{2m}} B^{k p^{2m}},
\end{equation}
\begin{equation}
\label{pip9}
(A A^{l p^m}B^{k p^m})^{\alpha}=A^{1+(l+\ell)p^m+\ell(l+kd-k)p^{2m}} B^{k p^{m}}.
\end{equation}
Applying the automorphism $A\leftrightarrow B$ to (\ref{pip9}) and (\ref{zim}) yields
\begin{equation}
\label{zim2}
(B B^{l p^m}A^{k p^m})^{p^m}=B^{p^m+(l+kd)p^{2m}} A^{k p^{2m}},
\end{equation}
\begin{equation}
\label{pip10}
(B B^{l p^m}A^{k p^m})^{\alpha}=B^{1+(l+\ell)p^m+\ell(l+kd-k)p^{2m}} A^{k p^{m}}.
\end{equation}

We readily see from (\ref{pip7}) and (\ref{zim2})
that the third and fourth defining relations of $J$ are always preserved. Moreover, (\ref{tuto3}) and (\ref{pip8}),
together with $\gcd(p,\ell)=1$, show that the first defining relation of $J$ is preserved if and only if
\begin{equation}
\label{roj}
i+j(2-d)+l\equiv 0\mod p^m.
\end{equation}
We also see from (\ref{tuto5}) and (\ref{pip10}), together with $\gcd(p,\ell)=1$, that the second defining relation of~$J$ is preserved if and only if
\begin{equation}
\label{roj2}
i+k(2-d)+l\equiv 0\mod p^m.
\end{equation}
Thus, (\ref{pupicat}) extends to an endomorphism, say $\Upsilon$, of $J$ if and only if (\ref{roj}) and (\ref{roj2}) hold,
in which case $\Upsilon$ is an automorphism. Indeed, (\ref{pip7}) ensures that the restriction of $\Upsilon$
to $Z(J)$ is the identity map and, since $J$ is nilpotent, this guarantees that $\ker(\Upsilon)$ is trivial. In particular,
setting $i=0$ and $j=1$, we infer the existence of the automorphism $\Delta$. Moreover,
(\ref{roj}) and (\ref{roj2}) imply that $k\equiv j\mod p^m$. Thus, if (\ref{roj}) and (\ref{roj2}) hold, then $\Upsilon$
is given by
\begin{equation}
\label{roj3}
A\mapsto AA^{i p^m}B^{j p^m},\; B\mapsto B B^{-i-j(2-d)p^m}A^{jp^m}.
\end{equation}
We now deduce from (\ref{roj3}) that (\ref{impsi}) holds.

Now $A^C=A^{1+\ell p^m}$ and $B^C=B^{1-\ell p^m+\ell^2 p^{2m}}$, with $\gcd(p,\ell)=1$,
and $\mathrm{Aut}_2(J)=\ker(\psi)$, so $(C\delta)^\psi$ and $(a^{p^m},b^{-p^m})$ generate the same
subgroup of $Z(K)^2$. As
$\Delta^\psi=(b^{p^{m}},b^{-(2-d)p^{m}} a^{p^{m}})$, we infer $\mathrm{Im}(\psi)=\langle C\delta,\Delta\rangle^\psi$,
which implies $\mathrm{Aut}_3(J)=\langle \Delta, C\delta\rangle\mathrm{Aut}_2(J)$. Given that
$(a^{p^m},b^{-p^m})$ and $(b^{p^{m}},b^{-(2-d)p^{m}} a^{p^{m}})$ generate a subgroup
of $Z(J)^2$ isomorphic to $(\Z/p^m\Z)^2$, we now deduce that
$\mathrm{Aut}_3(J)/\mathrm{Aut}_2(J)\cong (\Z/p^m\Z)^2$, and therefore
$|\mathrm{Aut}_3(J)|=|\mathrm{Aut}_2(J)|\times p^{2m}=p^{6m}$, by Proposition~\ref{autg}. Since
$\mathrm{Aut}_3(J)=\langle \Delta, C\delta\rangle\mathrm{Aut}_2(J)$, we infer
$\mathrm{Inn}(J)\mathrm{Aut}_3(J)=\langle \Delta\rangle\mathrm{Inn}(J)\mathrm{Aut}_2(J)$, whence
$\mathrm{Inn}(J)\mathrm{Aut}_3(J)/\mathrm{Inn}(J)\mathrm{Aut}_2(J)$ is cyclic.
As $|\mathrm{Inn}(J)\mathrm{Aut}_2(J)|=p^{7m}$ by Proposition \ref{autg2}, we are left
to show that $|\mathrm{Inn}(J)\mathrm{Aut}_3(J)|=p^{8m}$. But $|\mathrm{Inn}(J)|=p^{6m}$,
$|\mathrm{Aut}_3(J)|=p^{6m}$ and $|\mathrm{Inn}(J)\cap \mathrm{Aut}_3(J)|=|Z_4(J)/Z(J)|=p^{4m}$,
as required.
\end{proof}

Recall that $J$ has an automorphism $\theta$ that interchanges $A$ and $B$, and inverts $C$.

\begin{theorem}\label{autjfull} We have $\mathrm{Aut}_4(J)=\mathrm{Inn}(J)\mathrm{Aut}_3(J)$,
$\mathrm{Aut}(J)=\langle \theta\rangle \mathrm{Aut}_4(J)$, and $|\mathrm{Aut}(J)|=2p^{8m}$.
\end{theorem}

\begin{proof} As $C\in Z_4(J)$, with $A^B=AC$ and $B^A=BC^{-1}$, we infer
$\mathrm{Inn}(J)\subset \mathrm{Aut}_4(J)$. Likewise we find that $\mathrm{Inn}(H)\subset \mathrm{Aut}_3(H)$.
It follows from Proposition \ref{zi2} that
$\mathrm{Aut}_4(J)/\mathrm{Inn}(J)\mathrm{Aut}_3(J)$ imbeds into
$\mathrm{Aut}_3(H)/\mathrm{Inn}(H)\mathrm{Aut}_2(H)$. The latter group is trivial by Proposition \ref{auth2},
whence $\mathrm{Aut}_4(J)=\mathrm{Inn}(J)\mathrm{Aut}_3(J)$ has order $p^{8m}$ by Theorem \ref{autj3}.
A second application of Proposition \ref{zi2}
yields an imbedding $\mathrm{Aut}(J)/\mathrm{Aut}_4(J)\hookrightarrow \mathrm{Aut}(H)/\mathrm{Aut}_3(H)$,  which
has order 2 by Theorem \ref{auth6}, so
$\mathrm{Aut}(J)=\langle \theta\rangle \mathrm{Aut}_4(J)=\langle \theta\rangle \mathrm{Inn}(J)\mathrm{Aut}_3(J)$ has order $2p^{8m}$.
\end{proof}

In terms of group extensions, we have the normal series
$$
1\subset \mathrm{Aut}_1(J)\subset\mathrm{Aut}_2(J)\subset \mathrm{Aut}_3(J)\subset \mathrm{Aut}_4(J)\subset \mathrm{Aut}(J).
$$
The first factor $\mathrm{Aut}_1(J)\cong (\Z/p^m\Z)^2$ by Proposition \ref{autg}. The second factor
$\mathrm{Aut}_2(J)/\mathrm{Aut}_1(J)\cong Z_2(J)^2/Z(J)^2\cong (\Z/p^m\Z)^4$ also by Proposition \ref{autg},
which yields $\mathrm{Aut}_2(J)\cong (\Z/p^m\Z)^4$ directly as well. The third factor $\mathrm{Aut}_3(J)/\mathrm{Aut}_2(J)
\cong (\Z/p^m\Z)^2$
by Theorem \ref{autj3}. By Theorem \ref{autjfull}, the fourth factor
$$
\mathrm{Aut}_4(J)/\mathrm{Aut}_3(J)\cong \mathrm{Aut}_3(J)\mathrm{Inn}(J)/\mathrm{Aut}_3(J)\cong
\mathrm{Inn}(J)/\mathrm{Inn}(J)\cap \mathrm{Aut}_3(J)\cong J\delta/Z_4(J)\delta
$$
is isomorphic to $(J/Z(J))/(Z_4(J)/Z(J))\cong J/Z_J(H)\cong (\Z/p^m\Z)^2$. The fifth factor
$$
\mathrm{Aut}(J)/\mathrm{Aut}_4(H)\cong \Z/2\Z
$$
by Theorem \ref{autjfull}.

We next find a presentation of $\mathrm{Aut}_4(J)$, the sole Sylow $p$-subgroup of $\mathrm{Aut}(J)$. Note that
$\mathrm{Inn}(J)\cong J/Z(J)\cong H$ is generated by $a=A\delta$ and $b=B\delta$, subject to the defining
relations
\begin{equation}\label{puj}
a^{[a,b]}=a^\alpha, b^{[b,a]}=b^\alpha, a^{p^{2m}}=1, b^{p^{2m}}=1.
\end{equation}
Set $x=(A^{p^{2m}},A^{-p^{2m}})\Psi\in \mathrm{Aut}_2(J)$, as in Proposition \ref{autg}.
By Proposition \ref{autg2}, the group $\mathrm{Inn}(J)\mathrm{Aut}_2(J)=
\langle x\rangle\ltimes\mathrm{Inn}(J)$ is generated by $a,b,x$ subject to the defining relations (\ref{puj}) and
\begin{equation}\label{puj2}
x^{p^m}=1, a^x=a^{1+p^{2m}}=a, b^x=b^{1+p^{2m}}=b.
\end{equation}
By Theorem \ref{autj3}, $\Delta^{p^m}\in \langle x\rangle\ltimes\mathrm{Inn}(J)$. In fact, repeated application
of the definition of $\Delta$, with aid from (\ref{pip7}), (\ref{zim2}), and $(A^{p^{2m}})^\Delta=A^{p^{2m}}$,
yields $\Delta^{p^m}=(A^{-p^{2m}}, A^{(3-d)p^{2m}})\Psi$, where $d$ is defined as in Theorem \ref{autj3}.
On the other hand, we have $\alpha=1+tp^m$, with $t\in\Z$ and $\gcd(t,p)=1$, which implies
$[a,b]^{p^m}=C^{p^m}\delta=(A^{tp^{2m}}, A^{tp^{2m}})\Psi$. Thus
\begin{equation}\label{puj3}
\Delta^{p^m}=[a,b]^{up^m} x^v,
\end{equation}
where $u,v\in\Z$ are uniquely determined modulo $p^m$ by $-1\equiv ut+v,\; 3-d\equiv ut-v\mod p^m$.

Moreover, it is clear that $\mathrm{Inn}(J)$ is normalized by $\Delta$, with
\begin{equation}\label{puj4}
a^{\Delta}=a b^{p^m}, b^{\Delta}=b b^{-(2-d)p^m} a^{p^m}.
\end{equation}

Finally, $\mathrm{Inn}(J)\mathrm{Aut}_2(J)$ is also normalized by $\Delta$, so $x^\Delta\in \langle x\rangle\ltimes\mathrm{Inn}(J)$
and, in fact,
\begin{equation}\label{puj5}
x^{\Delta}=x,
\end{equation}
since $x\in \mathrm{Aut}_1(J)$, $x$ fixes $Z_3(J)$ pointwise, by Proposition \ref{autg}, $\Delta\in \mathrm{Aut}_3(J)$,
and $\Delta$ fixes $Z(J)$ pointwise, all of which readily implies $x\Delta=\Delta x$. We have proven the following result.

\begin{theorem} The group $\mathrm{Aut}_4(J)$ is a normal
Sylow $p$-subgroup of $\mathrm{Aut}(J)$ of order $p^{8m}$ and index 2 in $\mathrm{Aut}(J)$. It is generated by
$a=\delta A,b=\delta B,x=(A^{p^{2m}},A^{-p^{2m}})\Psi$, and $\Delta$, as given in Proposition \ref{autg} and Theorem \ref{autj3},
and subject to the defining relations (\ref{puj})-(\ref{puj5}).
\end{theorem}


\noindent{\bf Acknowledgments.} We thank Volker Gebhardt for useful discussions and his help with Magma,
and the referee for a careful reading of the paper and several useful suggestions.






\end{document}